\documentclass[a4paper, 12pt]{amsart}
\usepackage{latexsym, amssymb, amscd, amsfonts, amsthm, amsmath, graphicx}

\makeatletter
\@namedef{subjclassname@2010}{%
  \textup{2010} Mathematics Subject Classification}
\makeatother

\newtheorem{thm}{Theorem}[section]

\newtheorem{lem}[thm]{Lemma}

\newtheorem{prp}[thm]{Proposition}

\newtheorem{mainthm}[thm]{Main Theorem}

\theoremstyle{definition}

\numberwithin{equation}{section}

\newcommand{\supp}{{\mathrm{supp}}}

\newcommand{\cB}{{\mathcal{B}}}

\newcommand{\cS}{{\mathcal{S}}}

\newcommand{\zed}{\mathbb{Z}}

\newcommand{\bR}{\mathbb{R}}

\newcommand{\cO}{{\mathcal{O}}}
\newcommand{\sym}{{\mathrm{sym}}}

\begin{document}

\baselineskip=17pt

\title[Modular form $L$-functions]{Zero-density estimate for modular form $L$-functions in weight aspect}

\author[B. Hough]{Bob Hough}
\address{Department of Mathematics\\ Stanford University\\
450 Serra Mall, Stanford, CA, USA}
\email{rdhough@stanford.edu}

\date{February 6, 2012}

\begin{abstract}
We prove a zero-density estimate for the $L$-functions associated to modular forms of large fixed weight on the full modular group.  The estimate has applications to the distribution of central values, and to non-vanishing.
\end{abstract}

\maketitle
\section{Introduction}
In the analytic theory of $L$-functions, it is sometimes possible to circumvent
assumption of the Riemann Hypothesis by applying zero density
arguments.  Briefly, one argues that in a family of $L$-functions that is
sufficiently ``spectrally complete'', the number of zeros of functions
in the family to the right of the half-line is comparatively
few. 
Historically, zero density questions were first considered with respect to the
Riemann zeta function $\zeta(\frac{1}{2} + \sigma + it)$ as the parameter $t$
varied, and the first result along these lines could be said to be the
Hadamard-de la Val\'{e}e-Poussin zero-free region.  Later investigations
focused on the number 
\[N(\sigma, T) = \#\left\{\rho = \frac{1}{2} + \beta + i \gamma: \zeta(\rho) =
0,
\sigma < \beta,
0<\gamma < T\right\},\] proving that this number decayed in the power of $T$
with
increasing
$\sigma>0$.  A classical result in this direction
is due to Ingham \cite{ingham}
\[N(\sigma, T) = O(T^{3(\frac{1}{2}-\sigma)/(\frac{3}{2}-\sigma)}\log^5 T).\]
Selberg \cite{selberg_zeta} made a
major contribution to this theory, proving the uniform bound
\[N(\sigma, T) \ll T^{1-\frac{\sigma}{4}} \log T,\] 
in $0\leq \sigma \leq \frac{1}{2}$. The crucial feature of this estimate is
that the power of $\log T$ matches the true order in the number of zeros of
$\zeta$ up to height $T$, so that the estimate is still useful even when
$\sigma$ is on the order of $\frac{1}{\log T}$.  This  formed one of the
key analytic ingredients in
Selberg's unconditional proof that the real and imaginary parts of $\log
\zeta(\frac{1}{2} + it)$ become normally distributed in large intervals $t \in
[T,2T]$.

Subsequent to his work on $\zeta$, Selberg \cite{selberg_dirichlet} proved an
analogous zero density estimate in the family of Dirichlet $L$-functions to a
large modulus $q$, with $q$ rather than $t$ thought of as the varying parameter.
Using this estimate, he showed that for fixed $t$ the argument of
$L(\frac{1}{2} + it, \chi)$ becomes normally distributed as $\chi$ varies
modulo $q$, for $q \to \infty$.  More recently Luo \cite{luo} has given an
analogue of Selberg's bound in $t$-aspect,
replacing $\zeta$ with the $L$-function of a fixed Hecke-eigen cusp form for
$SL_2(\zed)$:
\begin{align*} N_f(\sigma,T) :=& \#\left\{\rho = \frac{1}{2} + \beta + i\gamma: L(\rho;f) =
0,
\sigma< \beta, 0 < \gamma < T\right\}\\& \ll_f T^{1 - \frac{\sigma}{72}} \log T.\end{align*}
Together with earlier work of Bombieri and Hejhal \cite{bombieri_hejhal}, this
established the asymptotic normality of  $\log L(\frac{1}{2}
+ it;f)$ for $t \in [T, 2T]$, $f$ fixed with $T \to \infty$.  

The purpose of this article
is to prove a parallel extension of Selberg's Dirichlet $L$-function
estimate but now for the family of  $L$-functions associated to modular forms
of  large weight $k$. As in Selberg's work, an
important aspect of our estimate  is
that it is uniform in $k$ and for $T$ in the ranges
$\frac{1}{\log k} < T <
k^{\delta},$ for some small $\delta > 0.$
This plays a crucial role in the author's related paper \cite{hough_normal},
where it is established, unconditionally, that varying $f$ among
Hecke-eigenforms of weight $k$, $\log L(\frac{1}{2};f)$ is bounded above
by a quantity that is asymptotically normal as $k \to \infty$.  One further
piece of context: Kowalski and Michel \cite{kowalski_michel} have proven
another extension of Selberg's theorem to the family of weight 2 modular forms
of large prime level $q$, and Conrey and Soundarajan \cite{conrey_sound}
(real Dirichlet $L$-functions) and Ricotta \cite{ricotta} (Rankin-Selberg
$L$-functions) have given related estimates, each with applications to
non-vanishing.  Suitably modified, our estimate has similar
applications, but we do not pursue them here.

To state our density result more precisely, let
$S_k$ denote the space of weight $k$ holomorphic cusp forms for the modular
group
$\Gamma = SL(2,\zed)$ and let $H_k$ be the basis of forms in $S_k$ that are
simultaneous eigenfunctions of all the Hecke operators.  Write the Fourier
expansion of $f \in H_k$ as 
\[f(z) = \sum_{n = 1}^\infty n^{\frac{k-1}{2}}\lambda_f(n) e(nz).\]  We
normalize
$f \in H_k$ so that $\lambda_f(1) = 1$.\footnote{In particular, in our
normalization
Deligne's bound \cite{deligne} reads $|\lambda_f(n)| \leq d(n)$, the number of
divisors of $n$.}  The $L$-function associated to $f \in H_k$ is
\begin{equation}\label{euler_product}
 L(s;f) = \sum_{n = 1}^\infty \frac{\lambda_f(n)}{n^s} = \prod_{p} \left(1 -
\frac{\lambda_f(p)}{p^s} + \frac{1}{p^{2s}}\right)^{-1}, \qquad\qquad \Re(s) >
1.
\end{equation}
This is a degree two $L$-function with completed $L$-function
\[\Lambda(s;f) = (2\pi)^{-s}
\Gamma(s + \frac{k-1}{2}) L(s;f)\]  satisfying the self-dual
functional equation 
\[\Lambda(s;f) = i^k \Lambda(1-s;f).\]  In particular, with our normalization
the Riemann Hypothesis
asserts that all zeros $\rho$ of $\Lambda(s;f)$ satisfy $\Re(\rho) =
\frac{1}{2}$.

For $f \in H_k$ and $T$ growing, but small compared to $\sqrt{k}$, the number of
zeros $\rho$ of $L(s;f)$ with $0 < \Im(\rho)<T$ is $\sim \frac{T}{2\pi}\log
k$.  Thus the density of zeros of $L(s;f)$ near the central point $s =
\frac{1}{2}$ is $\frac{\log k}{2\pi}$.  Our main result says that among the family $\mathcal{F}_k$ of $L$-functions associated to forms in $H_k$, there are very few $L$-functions with zeros with small
imaginary part and real part to the right of $\frac{1}{2} + \frac{C}{\log k}$.
\begin{mainthm}\label{zero_density}
 Let $ \frac{2}{\log k}< \sigma < \frac{1}{2}$.  For some sufficiently small 
$\delta, \theta > 0$ we have uniformly in $\frac{10}{\log k}< T <
k^{\delta}$
\begin{align*}N(\sigma, T) &\stackrel{\text{def}}{=} \frac{1}{|H_k|}{\sum_{f \in H_k}}
\#\left\{L(\frac{1}{2} + \beta + i \gamma) = 0: \sigma < \beta, \; |\gamma|<
T\right\}\\&= O(Tk^{ - \theta \sigma } \log k).\end{align*}
\end{mainthm}

The main new
analytic ingredient of our theorem is the following asymptotic evaluation of the
harmonic twisted second moment of $L(s;f)$, which may be of independent
interest.
\begin{thm}\label{twisted_moment}
 Let $\sigma > 0$, $0 \neq |t| < k^{\frac{1}{4}}$ and $\ell < k^{\frac{1}{3}}$
be squarefree. Denote by $\tau_\nu(n) = \sum_{n_1n_2 = n} \left(\frac{n_1}{n_2}\right)^\nu$  the generalized divisor function. We have the
following formula for the harmonic twisted
second moment.
\begin{align*}{\sum_{f
\in H_k}}^h &\lambda_f(\ell) |L(\frac{1}{2} + \sigma + it; f)|^2 \\=& \zeta(1 +
2\sigma)\frac{\tau_{it}(\ell)}{\ell^{\frac{1}{2} + \sigma}}
+
\zeta(1-2\sigma) 
\left(\frac{k}{4\pi}\right)^{-4\sigma}\frac{\tau_{it}(\ell)}{\ell^{\frac{1}{2} -
\sigma}}
\\&+ i^k\zeta(1 +
2it)\left(\frac{k}{4\pi}\right)^{ - 2\sigma +
2it}\frac{\tau_\sigma(\ell)}{\ell^{\frac{1}{2} +
it}} +
i^k\zeta(1 - 2it) 
\left(\frac{k}{4\pi}\right)^{ - 2\sigma -
2it}\frac{\tau_{\sigma}(\ell)}{\ell^{\frac{1}{2} -
it}} 
\\& + O\left(\ell^{3/4}k^{-1/2-2\sigma +
\epsilon}\right).
\end{align*}
\end{thm}
\noindent The harmonic average (${\sum}^h$) means that
forms $f \in H_k$ are counted with the weight
$w_f = \frac{(4\pi)^{1-k}\Gamma(k-1)}{\langle f,
f\rangle},$ which appears in the Petersson trace formula.
Harmonic averages similar to this
one have an extensive history;
see for instance \cite{kuznetsov}, \cite{faiziev}, \cite{fomenko} and
references therein. Our proof is most noteworthy for the fact that the
evaluation of main
 terms goes ``beyond the diagonal'' and yet is not too difficult. After
applying the Petersson trace formula and Voronoi summation to the resulting
sums of Kloosterman sums, the off-diagonal main term arises as the Fourier
transform of the relevant function at zero, and the remaining integrals against
Bessel functions are error terms.  The analysis of these error terms involves
integrating against the Bessel function $J_{k-1}(x)$ near its transition region,
and this is bounded in a similar way to an analysis of the twisted first moment
of $L(\frac{1}{2}, \sym^2 f)$ in \cite{khan}.

\section{Outline of proof}
The method of proof of Theorem
\ref{zero_density} is the same as in Selberg's original work on Dirichlet
$L$-functions; in particular, we appeal to the following version of the
argument principle introduced there.  
\begin{lem}\label{selberg_argument}
 Let $\omega$ be an entire function, non-zero in the half plane $\Re(s) >
W$. Let $\cB$ be the rectangular box $|\Im(s)|
\leq H$, $W_0 \leq \Re(s) \leq W_1$ with $W_0 < W< W_1$.  Then
\begin{align*}
 4H& \sum_{\substack{\beta + i \gamma \in \cB \\ \omega(\beta + i\gamma) = 0}}
\cos\left(\frac{\pi \gamma}{2H}\right)\sinh\left(\frac{\pi(\beta -
W_0)}{2H}\right) \\ 
&= \int_{-H}^H \cos\left(\frac{\pi t}{2H}\right)\log|\omega(W_0 + it)|
dt\\& \qquad-\Re\int_{-H}^H
\cos\left(\pi \frac{W_1 - W_0 + it}{2iH}\right)\log \omega(W_1 + it)dt\\ 
&\qquad+\int_{W_0}^{W_1}\sinh\left(\frac{\pi(\alpha -
W_0)}{2H}\right)\log|\omega(\alpha + iH)\omega(\alpha -
iH)|d\alpha. 
\end{align*}

\end{lem}
\noindent The fundamental proposition that we prove is the following.
\begin{prp}\label{natural_mollification} There exist  Dirichlet polynomials
 $\{M(s;f)\}_{f \in H_k}$, which satisfy $M(\overline{s}) = \overline{M
(s)}$ and such that for sufficiently small positive $\delta$ and $\theta$,
uniformly in $|t| < k^\delta$, $\frac{1}{\log k}\leq \sigma \leq 1$
\[\frac{1}{|H_k|} \sum_{f \in H_k} \left|M(\frac{1}{2} +\sigma+ it; f)
L(\frac{1}{2} + \sigma +it; f)\right|^2 \leq 1 + O(k^{-\theta
\sigma})\]
and for all $t$,
\[M(\frac{3}{2} + it;f)L(\frac{3}{2} + it; f) = 1 + O(k^{-\theta}).\]
\end{prp}
\noindent To deduce Theorem \ref{zero_density} from this Proposition, apply the
lemma to  $M\cdot L(s;f)$ with box bounded by
$\frac{1}{2} + \frac{1}{\log k} \pm 2iT$ and $\frac{3}{2} \pm 2iT$.  The
special feature of the lemma which permits uniformity even for small $T
\asymp \frac{1}{\log k}$, is that only the real part of the logarithm appears
in the part of the integral contained in the critical strip, so that this
part may be bounded using the second moment estimate of the Proposition. 
The further details of the deduction of Theorem \ref{zero_density} are
not difficult, and may be found both in Selberg's original argument, and in
the treatments in \cite{conrey_sound} and \cite{kowalski_michel}.  In the
remainder of the paper we are concerned with the proof of the Proposition,
which takes place in three stages: first we calculate the harmonic twisted
moment, proving Theorem \ref{twisted_moment}.  Next we mollify the second moment
with respect to the harmonic weights.  Finally we remove the harmonic weights
via the method of \cite{kowalski_michel}. 

\section{Some lemmas}\label{lemma_section}

\begin{lem}[Hecke Relations] \label{hecke}
 For each Hecke eigenform $f$, the Fourier coefficients of $f$ satisfy the
relation
\[\lambda_f(m)\lambda_f(n) = \sum_{d|(m,n)}
\lambda_f\left(\frac{mn}{d^2}\right).\]
\end{lem}
\noindent This is equivalent to the Euler product (\ref{euler_product}).

The basic orthogonality relation on $H_k$ is the Petersson Trace Formula.
\begin{lem}[Petersson Trace Formula] \label{petersson}
 We have
\[{\sum_{f \in H_k}}^h \lambda_f(m)\lambda_f(n) = \delta_{m = n} + 2\pi i^k
\sum_{c = 1}^\infty \frac{S(m,n;c)}{c} J_{k-1}\left(\frac{4\pi}{c}
\sqrt{mn}\right)\]
\end{lem}
\begin{proof}
 See e.g. \cite{iwaniec_kowalski} p 360.
\end{proof}

Recall that we denote by \begin{equation}\label{divisor_notation}\tau_{\nu}(n) =
\sum_{n_1n_2=n}\left(\frac{n_1}{n_2}\right)^\nu\end{equation} the generalized
divisor function.  We will use the following version of the Voronoi summation
formula.

\begin{lem}\label{voronoi} Let $g: \bR^+ \to \bR^+$ be smooth with compact
support. Let $c \geq 1$ and $(a,c) = 1$ with $ad \equiv 1 \bmod c$.    We have
\begin{align*}
 &\sum_{m=1}^\infty \tau_{it}(m) e\left(\frac{am}{c}\right) g(m) \\& \qquad= c^{2it
-1}\zeta(1 - 2it) \int_0^\infty g(x)x^{-it}dx+ c^{-2it -1}\zeta(1 + 2it)
\int_0^\infty g(x)x^{ it} dx  \\& \qquad\qquad + \frac{1}{c} \sum_{n
=1}^\infty \tau_{it}(n) e\left(\frac{-dn}{c}\right) \int_0^\infty g(x)
J_{2it}^+\left(\frac{4\pi}{c}\sqrt{nx}\right) dx \\ & \qquad\qquad +
\frac{1}{c}\sum_{n=1}^\infty \tau_{it}(n)e\left(\frac{dn}{c}\right)\int_0^\infty
g(x) K_{2it}^+\left(\frac{4\pi}{c}\sqrt{nx}\right) dx
\end{align*}
where
\begin{align*} J_{\nu}^+(x) &= \frac{-\pi}{\sin \frac{\pi \nu}{2}}(J_{\nu}(x) -
J_{-\nu}(x)), \qquad  K_{\nu}^+(x)  = 4 \cos \frac{\pi \nu}{2} K_\nu(x).
\end{align*}
\end{lem}

\begin{proof}
 This is a slight modification of \cite{iwaniec_kowalski} Theorem 4.10.
\end{proof}

In bounding oscillatory integrals we make use of the following simple estimate
(\cite{titchmarsh}, Lemma 4.5)
\begin{lem}\label{oscillatory_bound}
 Let $F(x), G(x)$ be real-valued functions on $[a,b]$ satisfying
$\frac{F'(x)}{G(x)}$ is monotonic and $F''(x) > r > 0$, $|G(x)| \leq M$.  Then 
\[\left|\int_a^b G(x) e^{i F(x)}dx \right| \leq \frac{8M}{\sqrt{r}}.\]
\end{lem}

\subsection{Facts concerning Bessel functions}
Bessel functions arise both in the Petersson Trace Formula and as transforms
in the Voronoi Summation Formula; we record here the properties that we will
need regarding these functions.  

The Bessel function of the first kind,
$J_\nu(x)$, has Taylor series about zero
given by 
\begin{equation*}
J_{\nu}(x) = \sum_{m = 0}^\infty \frac{(-1)^m
\left(\frac{z}{2}\right)^{\nu+
2m}}{m! \Gamma(\nu+1+m)}.
\end{equation*}
Differentiating, one obtains the relation 
\begin{equation}\label{differentiation_relation} J_\nu'(x) =
\frac{1}{2}\left(J_{\nu+1}(x) - J_{\nu-1}(x) \right).\end{equation}
Specializing to $\nu = k-1$, the Mellin Transform is given by
\begin{equation}\label{mellin_transform}\int_0^\infty J_{k-1}(x) x^{s-1}dx =
2^{s-1}\frac{\Gamma\left(\frac{k-1+s}{2}\right)}{\Gamma\left(\frac{k+1-s}{2}
\right )}.\end{equation}  

The behavior of all of the Bessel functions depends essentially on the
relationship between the size of the order $\nu$ and the variable $x$.  
When $x$ is large,  $x > |\nu|^2$ ($\nu$ possibly complex)  $J_\nu$
is oscillatory of essentially fixed frequency, while
the Bessel function of the third kind $K_\nu$ is exponentially small. 
Asymptotic
evaluations are given
by 
(\cite{bateman_higher_vol_2} p. 85)
\begin{align}\label{J_large_x} J_\nu(x) &= \sqrt{\frac{2}{\pi x}} \cos\left(x -
\frac{\pi \nu}{2} - \frac{\pi}{4}\right) \left[1 - \frac{P(\nu) }{128
x^2}\right] \\ \notag &\qquad - \sin\left(x -
\frac{\pi \nu}{2} - \frac{\pi}{4}\right) \frac{\nu^2 - \frac{1}{4}}{2x} +
O\left(\frac{1 +|\nu|^6}{x^3}\right)\\
\label{J_+_large_x}
J_\nu^+(x) &= -\sqrt{\frac{2\pi}{x}} \sin\left(x -
 \frac{\pi}{4}\right) \left[1 - \frac{P(\nu) }{128
x^2}\right] \\&\notag \qquad-\pi \cos\left(x - \frac{\pi}{4}\right) \frac{\nu^2 -
\frac{1}{4}}{2x} +
O\left(\frac{1 +|\nu|^6}{x^3}\right)
\end{align}
where $P(\nu) = 16\nu^4 -40\nu^2+9$.
Also,
\begin{equation} \label{K_large_x}  K_{\nu}(x) = \sqrt{\frac{\pi}{2x}}e^{-x}\left[1 +
O\left(\frac{1+|\nu|^2}{x}\right)\right].
\end{equation}
Since we regard $t$ as small compared to $k$, $|t| < k^{1/4}$, these are the
only evaluations we need regarding $J_{it}$, and $K_{it}$.

When $x$ is small, $x \ll k$, $J_{k-1}(x)$ is uniformly small.  Taking absolute
values in the Taylor expansion leads to the bound
(\cite{rudnick_sound_examples}, p 5)
\begin{equation}\label{small_z_bound} |J_{k-1}(x)| \leq
\frac{(\frac{x}{2})^{k-1}}{\Gamma(k-1)}e^{\frac{x}{2}}, \qquad\qquad
x < 2k.\end{equation}  In particular, if $x < \frac{k}{10}$ then $J_{k-1}(x) <
e^{-k}$.  

In the transition region $k \ll x \ll k^2$, $J_{k}(x)$ increases to a
global maximum of size
$k^{-1/3}$ at a point near $x = k$, and thereafter oscillates
with slowly increasing frequency and slowly decreasing amplitude.  Langer's
Formulas (\cite{bateman_higher_vol_2} p 85, (32) and (34)) give an asymptotic
evaluation: 
\begin{align}\label{langer_small_x}
 J_k(x) = &\frac{(\tanh^{-1} w - w)^{1/2}}{\pi w^{1/2}} K_{1/3}(z) +
O(k^{-4/3}); \\\notag  & \qquad \qquad x < k, \quad w = \left(1 -
\frac{x^2}{k^2}\right)^{1/2}, \quad z = k(\tanh^{-1}w - w)
\\ \label{langer_large_x}
J_k(x) = &\frac{(w - \tan^{-1}w)^{1/2}}{w^{1/2}}[J_{1/3}(z)\cos(\pi/6) -
Y_{1/3}(z)\sin(\pi/6)] + O(k^{-4/3}); \\ \notag  & \qquad\qquad x> k, \quad
w=\left(\frac{x^2}{k^2} - 1\right)^{1/2}, \quad  z = k(w - \tan^{-1}w).
\end{align}
Here $Y_{\nu}(x)$ is the Bessel function of the second kind, related to $J_\nu$
by
\[ Y_\nu(x) = \frac{J_\nu(x)\cos(\nu \pi) - J_{-\nu}(x)}{\sin \nu \pi}.\]
Since Langer's formulas depend on the functions $J_{1/3}$, $Y_{1/3}$ and
$K_{1/3}$ we record their further asymptotic properties.  For $x\gg 1$ the
evaluations of $J_{1/3}$ and $K_{1/3}$ are given by (\ref{J_large_x}) and
(\ref{K_large_x}), while the evaluation of $Y_{1/3}$ is the same as for
$J_{1/3}$ except that the places of $\cos$ and $\sin$ are interchanged.  When
$x <1$ we have the bounds
\begin{equation}\label{small_x_bounds} J_{1/3}(x) \ll
x^{1/3}, \qquad  Y_{1/3}(x) \ll x^{-1/3}, \qquad  K_{1/3}(x) \ll
x^{-1/3}.\end{equation}
We collect together these facts in the following lemma.
\begin{lem}\label{transition_region_facts}
 In the region $|x-k| < k^{1/3}$ we have the bound
\begin{equation}\label{central_region} J_k(x) \ll k^{-1/3}.
\end{equation}
For $0<x < k - k^{1/3}$ we have
\begin{equation}
 \label{just_below} J_k(x) = \frac{e^{k(w -
\tanh^{-1}w)}}{\sqrt{2\pi k w}}\left[1 + O(k^{-1}w^{-3}) \right]+O(k^{-4/3}),
\end{equation} with  $w = \left(1 - \frac{x^2}{k^2}\right)^{1/2}.$
For $x > k + k^{1/3}$ we have
\begin{align}\label{just_above}
 J_k(x)  &=  \sqrt{\frac{2}{\pi k
w}}\cos\left(k(w - \tan^{-1} w) - \frac{\pi }{4} \right) +O\left(k^{-4/3} +
\frac{1 + w^{-2}}{k^{3/2}w^{3/2}}\right)\end{align} with, now, $w =
\left( \frac{x^2}{k^2} - 1\right)^{1/2}.$
\end{lem}

\begin{proof}
 Note that for $x = k \pm k^\Delta$ and  $\Delta < 1$, $w \asymp
k^{(\Delta-1)/2}$. Thus for $|x-k| < k^{1/3}$ the bound follows from
Langer's formulas and the bounds in (\ref{small_x_bounds}).  For $|x-k| >
k^{1/3}$ we have $w \gg k^{-1/3}$ and, therefore, $z \gg 1$.  The remaining
formulas thus follow from the asymptotic evaluations of $J_{1/3}$, $Y_{1/3}$,
$K_{1/3}$ at large argument, together with Langer's formulas.
\end{proof}

One further consequence is the following simple lemma.
\begin{lem}\label{average_size_bound}
 For any integer $k > 0$ and any $A < k^2$,
\[\int_0^A |J_k(x)| dx \ll \sqrt{A}.\]
\end{lem}
\begin{proof}
In the range $A < k -
k^{1/2 }$ the formula (\ref{just_below}) and $w \gg k^{-1/4}$
imply uniformly
$J_k(x) \ll k^{-4/3} + e^{-\Omega(k^{1/4})}k^{O(1)}$.  For
$k-k^{1/2} < x < k + k^{1/2}$ bound simply
$J_k(x) = O(1)$.  In the range $k + k^{1/2}< A < 2k$ use
(\ref{just_above}) to bound 
\[\int_{k + k^{1/2}}^A |J_k(x)|dx \ll \frac{A}{k^{4/3}} + \int_{k +
k^{1/2}}^A \left(\frac{1}{\sqrt{kw}} + \frac{1}{k^{3/2}w^{7/2}}\right) dx.\]  For $k
< x
< 2k$, $w \gg \sqrt{\frac{x-k}{k}}$, so the last integral is
\[\ll \int_{k^{1/2}}^{A-k}\left( \frac{1}{(ky)^{1/4}} +
\frac{k^{1/4}}{y^{7/4}}\right) dy
\ll \frac{A^{3/4}}{k^{1/4}} + k^{-3/8} \ll \sqrt{A}.\]  Finally,
for $x > 2k,$ $w = \Omega(1)$ and so (\ref{just_above}) says that $|J_k(x)| \ll
\frac{1}{\sqrt{x}} + \frac{1}{k^{4/3}}$, which plainly suffices.
\end{proof}

With an eye toward applying Lemma \ref{oscillatory_bound} and with $x > k$ and $
w = \sqrt{
\frac{x^2}{k^2}-1}$ as above, we record
\begin{equation}\label{phase_change}\frac{\partial}{\partial
x}(kw - k \tan^{-1}w) = \frac{kw}{x}, \qquad 
\qquad \frac{\partial^2}{ \partial x^2} (kw -
k\tan^{-1} w) = \frac{k^2}{x^2(x^2-k^2)^{1/2}}.\end{equation}

\subsection{Approximate functional equation} 

 Fix, once and for all, a smooth function $H:\bR^+ \to
\bR^+$ satisfying 
\begin{enumerate}
 \item $H(x) \equiv 1$ for $x \in [0, \frac{1}{2}]$
 \item $H(x) + H(\frac{1}{x}) = 1$.
\end{enumerate}
In particular, the Mellin transform $\hat{H}(s)$ has a single simple pole at 0
of residue 1, is odd, and satisfies the bounds
\[\hat{H}(s) \ll_A \frac{1}{s(s+1)...(s+A-1)}, \quad A = 1, 2, ...\] and, for
$\Re(s)>1$, $|\hat{H}(s)| \ll 2^{\Re(s)}.$

We  record an approximate formula  for $|L(\frac{1}{2} +
\sigma + it;f)|^2$. 
\begin{prp}[Approximate functional equation]\label{approx}
 We have
\begin{align}\label{AFE}\left|L\left(\frac{1}{2} + \sigma + it;f\right)\right|^2
&= \sum_{d =
1}^\infty \frac{1}{d^{1 + 2\sigma}}\sum_{m=
1}^\infty
\frac{\lambda_f(m) \tau_{it}(m)}{m^{\frac{1}{2} + \sigma }}\\\notag &
\quad\quad\times \left(W_{k,
\sigma+it}(md^2)  + (4\pi^2md^2)^{2\sigma} \tilde{W}_{k,
-\sigma + it}(md^2)\right)\end{align}
with 
\[W_{k, \sigma+it}(\xi) = \frac{1}{2\pi i} \int_{(3)}\frac{\hat{H}(s)}{(4\pi^2 \xi)^s}
\frac{\Gamma(\sigma + \frac{k}{2} + it + s)\Gamma(\sigma + \frac{k}{2} -it +
s)}{\Gamma(\sigma + \frac{k}{2} + it) \Gamma(\sigma + \frac{k}{2} -it)}
ds\]
and
\[\tilde{W}_{k, -\sigma + it}(\xi) = \frac{1}{2\pi i} \int_{(3)}
\frac{\hat{H}(s)}{(4\pi^2 \xi)^s}
\frac{\Gamma(-\sigma + \frac{k}{2} + it + s)\Gamma(-\sigma + \frac{k}{2} -it +
s)}{\Gamma(\sigma + \frac{k}{2} + it) \Gamma(\sigma + \frac{k}{2} -it)}
ds.\]
\end{prp}

\begin{proof}
See \cite{iwaniec_kowalski} pp 97-100.
\end{proof}

The functions   $W_{k,
\sigma + it}$
and $\tilde{W}_{k, -\sigma + it}$ satisfy the following properties.

\begin{lem} 
As functions of a real variable, both $W_{k, \sigma + it}$ and $\tilde{W}_{k,
-\sigma + it}$ are real valued.  For $t < k^{1/4}$ and $|\sigma| < 2$ we
have
\begin{align*} 
& W_{k, \sigma + it}(\xi) = 1 +
O\Bigl(\left(\frac{400 \xi }{k^2}\right)^{k^{1/4}}\Bigr), \qquad 
 W_{k, \sigma + it}(\xi) = O\Bigl(\left(\frac{
k^2}{80\xi}\right)^{k^{1/4}}\Bigr), \\&\qquad \qquad
\xi^j\left(\frac{\partial}{\partial
\xi}\right)^j W_{k, \sigma + it}(\xi)\ll_j
1
\\& \quad \text{and}\\
&\tilde{W}_{k, -\sigma + it}(\xi) =
\frac{\Gamma(-\sigma + \frac{k}{2} + it)\Gamma(-\sigma + \frac{k}{2}
-it)}{\Gamma(\sigma + \frac{k}{2} + it)\Gamma(\sigma + \frac{k}{2} - it)} +
O\Bigl(\left(\frac{400 \xi }{k^2}\right)^{k^{1/4}}\Bigr), \quad
  \\&\qquad  \tilde{W}_{k, -\sigma + it}(\xi) = O\Bigl(\left(\frac{
k^2}{80\xi}\right)^{k^{1/4}}\Bigr), \qquad 
\xi^j\left(\frac{\partial}{\partial \xi}\right)^j\tilde{W}_{k, \sigma +
it}(\xi)\ll_j k^{ -
4\sigma}.
\end{align*}
\end{lem}

\begin{proof}
Pair $s$ and $\overline{s}$ in the defining integrals to prove that $W$ and
$\tilde{W}$ are real. 

For the bounds on the functions,  shift the contour to $\Re(s) = \pm
k^{1/4}$ and estimate the ratio
of Gamma factors using Stirling's approximation.  In particular, for $|\Delta|<
\Re(z)^{1/2}$ we use the estimate
\begin{equation}\label{gamma_ratio}\frac{\Gamma(z + \Delta)}{\Gamma(z)} =
\exp\left(\Delta \log z + \frac{
\Delta^2}{2 z} + O\left(|\Delta||z|^{-1} \right) \right).\end{equation}

The derivatives are bounded
by estimating directly on the $\Re(s) = 0$ line.  

\end{proof}

\section{Twisted second moment, Proof of Theorem
\ref{twisted_moment}}\label{twisted_section}

From the approximate functional equation, 
\begin{align*}
 &{\sum_{f
\in H_k}}^h \lambda_f(\ell) \left|L\left(\frac{1}{2} + \sigma + it;
f\right)\right|^2\\ &\qquad= \sum_{m,d
= 1}^\infty \frac{\tau_{it}(m)}{m^{\frac{1}{2} + \sigma}d^{1 +
2\sigma}}\left[W_{k, \sigma +it}(md^2) + (4\pi^2
md^2)^{2\sigma}\tilde{W}_{k,-\sigma + it}(md^2)\right]\\ &\qquad\qquad\qquad\times {\sum_{f \in
H_k}}^h \lambda_f(\ell)\lambda_f(m)
\end{align*}

Applying the Petersson inner product
we obtain a diagonal term
\begin{equation*}\tag{D}\label{diagonal}
 \frac{\tau_{it}(\ell)}{\ell^{1/2 + \sigma}}\sum_{d = 1}^\infty
\frac{1}{d^{1 + 2\sigma}} \left[W_{k, \sigma + it}(\ell d^2) + (4\pi^2 \ell
d^2)^{2\sigma}\tilde{W}_{k, -\sigma + it}(\ell d^2)\right]
\end{equation*}
and an off-diagonal term 
\begin{align*}\tag{OD}\label{off_diagonal}&2\pi i^k \sum_{m,d = 1}^\infty 
\frac{\tau_{it}(m)}{m^{1/2 + \sigma}d^{1 + 2\sigma}} \left[W_{k, \sigma
+ it}(md^2) +
(4\pi^2 md^2 )^{2\sigma}\tilde{W}_{k, -\sigma + it}(md^2)\right]\\ &
\qquad\qquad\times
\sum_{c=1}^\infty
\frac{S(m,\ell;c)}{c}J_{k-1}\left(\frac{4\pi}{c}\sqrt{\ell
m}\right).\end{align*}
Introducing the integrals defining $W$ and $\tilde{W}$ the diagonal terms are
given by
\begin{align*}(\ref{diagonal}) = \frac{\tau_{it}(\ell)}{\ell^{1/2 + \sigma}} \Biggl\{
\frac{1}{2\pi i}
\int_{(3)} &\frac{\zeta(1 + 2s)}{(4\pi^2 \ell)^{s-\sigma}} \frac{\Gamma(s +
\frac{k}{2} + it)\Gamma(s + \frac{k}{2} - it)}{\Gamma(\sigma + \frac{k}{2} +
it)\Gamma(\sigma +
\frac{k}{2} - it)}\\&\qquad\qquad\times \left(\hat{H}(s-\sigma) + \hat{H}(s + \sigma)\right)
ds\Biggr\},\end{align*}
and this evaluates to the first two main terms, with an error that is
$O(\frac{1}{k})$, by shifting the contour to the line $\Re(s) =
-\frac{1}{2}+\sigma$.\footnote{The pole of $\zeta$ does not contribute since
$\hat{H}(-\sigma) + \hat{H}(\sigma) = 0$.}

The two remaining main terms come from off the diagonal, so we now work to
isolate these terms.  The crucial fact in evaluating the off-diagonal terms is
that the summations over $c$ and $d$ are very short.  Exchanging order of
summation, we write 
\begin{align*}&(\ref{off_diagonal})= 2\pi i^k\sum_{cd <
10000\sqrt{\ell}}\frac{\Sigma_{c,d}}{cd^{1 + 2\sigma}} +O\left( \sum_{cd \geq
10000\sqrt{\ell}} \frac{|\Sigma_{c,d}|}{cd^{1 + 2\sigma}}\right);\\
&\Sigma_{c,d} = \sum_{m = 1}^\infty
\frac{\tau_{it}(m)S(m,\ell;c)}{m^{1/2 + \sigma}}
J_{k-1}\left(\frac{4\pi}{c}\sqrt{\ell m}\right)\\& \qquad\qquad \times \left[W_{k, \sigma
+ it}(md^2) +
(4\pi^2 md^2 )^{2\sigma}\tilde{W}_{k, -\sigma +
it}(md^2)\right].
\end{align*}
We show that the sum over large $cd$ is an error that is $o(1)$. 
\begin{lem}\label{large_cd_bound}
 When $cd \geq 10000 \sqrt{\ell}$ we have the bound
\begin{align*}
\Sigma_{c,d} \ll \frac{(cdk\ell)^{O(1)}}{\Gamma(k-1)} e^{4\pi k
\sqrt{\frac{\ell^{1/2}}{cd}}} \left(2\pi k
\sqrt{\frac{\ell^{1/2}}{cd}}\right)^{k-1} + 
cd^{4\sigma}
\left(\frac{\ell^{1/2}}{80 cd}\right)^{k^{1/4}}.
\end{align*}
\end{lem}
\noindent This sufficies, since when summed over $cd >
10000\sqrt{\ell}$, the bound of the lemma yields 
\[\ll \ell^{O(1)}e^{-k^{1/4}} = o(1). \] 
\begin{proof}[Proof of lemma]
Split the sum over $m$ according as $m \leq 
\frac{k^2c}{\ell^{1/2} d}$, or not.  For small $m$, each term in the sum
is bounded by applying the
bound (\ref{small_z_bound}) for the Bessel function, bounding the Kloosterman
sum trivially by $c$ and bounding $W$ and
$\tilde{W}$ by $O(1)$.  This yields 
\begin{align*}&\ll cd^{4\sigma}\frac{1}{\Gamma(k-1)}\sum_{m < k^2
c/d\ell^{1/2}}e^{\frac{4\pi}{c}\sqrt{\ell m}}\left(\frac{2\pi}{c}
\sqrt{\ell m}\right)^{k-1} \\&\leq \frac{(cdk\ell)^{O(1)}}{\Gamma(k-1)} e^{4\pi k
\sqrt{\frac{\ell^{1/2}}{cd}}} \left(2\pi k
\sqrt{\frac{\ell^{1/2}}{cd}}\right)^{k-1},\end{align*} by bounding each term in
the sum by the largest term.  In the part of the sum with $m > \frac{k^2
c}{\ell^{1/2}d}$ we bound the Bessel function by $O(1)$, the Kloosterman
sum by $c$ and $W,
\tilde{W}$ by $\ll \left(\frac{ k^2}{80 d^2m}\right)^{k^{1/4}}$,
which gives
\[ \ll cd^{4\sigma} \sum_{m > k^2c/\ell^{1/2}d} \frac{1}{m^{1/2
+ \sigma}} \left(\frac{k^2}{80 md^2}\right)^{k^{1/4}} \ll cd^{4\sigma}
\left(\frac{\ell^{1/2}}{80 cd}\right)^{k^{1/4}}.\]  
\end{proof}

In order to apply the Voronoi summation formula to the sum over $m$ we open the
Kloosterman sum and introduce a function of compact support.  Let $F
\in C_c^\infty(\bR^+)$ satisfy
\begin{enumerate}
 \item $F(x) \equiv 1$ for $\frac{k}{1000} < x < 1000 k\sqrt{\ell}$
\item $\supp(F) \subset [\frac{k}{2000}, 2000 k\sqrt{\ell}]$
\item For each $j=0, 1, 2, ...$ and all $x$, $x^j \frac{d^j}{dx^j} F(x) \ll_j
1.$
\end{enumerate}
and consider the perturbed sum
\begin{align*}\tilde{\Sigma}_{c,d}= 
{\sum_{a \bmod c}}^* e\left(\frac{\overline{a}\ell}{c}\right)
\sum_{m = 1}^\infty &\frac{\tau_{it}(m)e(\frac{am}{c})}{m^{1/2 + \sigma}}
J_{k-1}\left(\frac{4\pi}{c}\sqrt{\ell
m}\right)F\left(\frac{4\pi}{c}\sqrt{\ell m}\right)\\ 
&\times\left[W_{k, \sigma
+ it}(md^2) +
(4\pi^2 md^2 )^{2\sigma}\tilde{W}_{k, -\sigma +
it}(md^2)\right] .
\end{align*}
This negligibly changes the sum, since for those $c,m$ for which
$F\left(\frac{4\pi}{c}\sqrt{\ell m}\right)$ is not identically 1, either $J_k$
 or $W$ or $\tilde{W}$ is extremely small: there are $O(\ell^{O(1)}k^{O(1)})$
terms with $\frac{4\pi}{c}\sqrt{\ell m} < \frac{k}{1000}$ and for these terms, the Bessel function is bounded by
 $e^{-k}$.  Meanwhile, if
$\frac{4\pi}{c}\sqrt{\ell m} > 1000 k\ell$ then $m >
\left(\frac{1000}{4\pi}\right)^2 k^2c^2$ so that the sum is bounded by 
\[\ll \ell^{O(1)}k^{O(1)} \sum_{cd < 1000\sqrt{\ell}} \sum_{m >
\left(\frac{1000}{4\pi}\right)^2 k^2c^2}
\left(\frac{k^2}{80md^2}\right)^{k^{1/4}} \ll \ell^{O(1)}k^{O(1)}
e^{-k^{1/4}}.\]

Introduce functions
\begin{align*} g_{c,d}(x) &= \frac{1}{x^{1/2 + \sigma}} W_{k, \sigma +
it}(d^2
x)J_{k-1}\left(\frac{4\pi}{c} \sqrt{\ell x}\right) F\left(\frac{4\pi}{c}
\sqrt{\ell x}\right),\\ \tilde{g}_{c,d}(x) &= \frac{1}{x^{1/2
- \sigma}} \tilde{W}_{k, -\sigma + it}(d^2
x)J_{k-1}\left(\frac{4\pi}{c} \sqrt{\ell x}\right) F\left(\frac{4\pi}{c}
\sqrt{\ell x}\right)\end{align*}
so that
\[\tilde{\Sigma}_{c,d} = {\sum_{a\bmod c}}^*e\left(\frac{\overline{a}\ell}{c}
\right)\sum_m \tau_{it}(m) e\left(\frac{am}{c}\right)\left\{g_{c,d}(m) + (4\pi^2
d^2)^{2\sigma}\tilde{g}_{c,d}(m)\right\}.\]  Applying, for each $c,d$, Voronoi
summation in the sum over $m$, we express the off-diagonal terms
as\footnote{Note that summation over ${a \bmod c}^*$ has been replaced by
Ramanujan sums.}
\begin{align*}
\frac{(\ref{off_diagonal})}{2\pi i^k}+o(1) =    &\zeta(1 - 2it)
\sum_{cd <
10000\ell^{1/2}} \frac{S(0,
\ell;
c)}{c^{2 - 2it} d^{1 + 2\sigma}} \int_0^\infty
g_{c,d}(x)x^{-it}dx 
\\&  +  \zeta(1 + 2it)\sum_{cd < 10000 \ell^{1/2}} \frac{S(0, \ell;c)}{c^{2
+ 2it} d^{1 + 2\sigma}} 
\int_0^\infty g_{c,d}(x)x^{ it} dx  
\\\tag{J}\label{J_sum} + 
\sum_{cd <
10000 \ell^{1/2}}&\sum_{n
=1}^\infty \frac{\tau_{it}(n) S(0, \ell-n;c)}{c^2d^{1 + 2\sigma}}   
\int_0^\infty g_{c,d}(x)
J_{2it}^+\left(\frac{4\pi}{c}\sqrt{nx}\right) dx 
\\\tag{K} \label{K_sum} +
\sum_{cd < 10000 \ell^{1/2}}&\sum_{n=1}^\infty \frac{\tau_{it}(n)S(0, \ell +n;
c)}{c^2 d^{1 +
2\sigma}}
\int_0^\infty
g_{c,d}(x) K_{2it}^+\left(\frac{4\pi}{c}\sqrt{nx}\right) dx
\\& \quad+ \text{ analogous terms coming from $\tilde{g}$}.
\end{align*}
 We
are going to show that the first two terms 
combine with the corresponding 
terms from $\tilde{g}$ to yield the remaining two main terms of the theorem,
and that (\ref{J_sum}) and (\ref{K_sum}) are error terms.

\subsection{The off-diagonal main terms}
Expanding the definition of $g_{c,d}(x)$, the first two terms above are equal to
\begin{align} \label{off_diagonal_main} 2i^k \Re\Biggl\{& 2 \pi \zeta(1 - 2it)
\sum_{cd <
10000\ell^{1/2}}
\frac{S(0, \ell;
c)}{c^{2 - 2it} d^{1 + 2\sigma}} \\ \notag &\times \int_0^\infty
 W_{k, \sigma + it}(d^2
x)J_{k-1}\left(\frac{4\pi}{c} \sqrt{\ell x}\right) F\left(\frac{4\pi}{c}
\sqrt{\ell x}\right)x^{-1/2 - \sigma -it}dx.\Biggr\}
\end{align}  With negligible
error the function $F$ may be removed from the integrand, and then the sums
extended to all $c$ and $d$, this justified by  the continuous analog of
the arguments given above involving summations over
$m$.\footnote{We bound only the real part of the error.  Recall that  $W$ and
$\tilde{W}$ are real, so that the imaginary parts of $c^{it}$ and $x^{it}$ are
 $O(t\log \ell)$ and $O(t\log x)$.} Inserting the definition
of $W_{k, \sigma+it}$ we obtain for the integral in (\ref{off_diagonal_main})
with $F$ removed
\begin{align*}&\int_0^\infty
\Biggl\{\left[\frac{1}{2\pi i}
\int_{(3)} \frac{\hat{H}(s)}{(4\pi^2 xd^2)^s} \frac{\Gamma(s + \sigma +
\frac{k}{2}
+it)\Gamma(s + \sigma + \frac{k}{2} - it)}{\Gamma(\sigma + \frac{k}{2} +
it)\Gamma(\sigma + \frac{k}{2}-it)} ds \right]
\\&\qquad\qquad\qquad\qquad\qquad\qquad\qquad\qquad\qquad\times
J_{k-1}\left(\frac{4\pi}{c}\sqrt{\ell x}\right)x^{1/2 -
\sigma-it}\Biggr\}\frac{
dx}{x}\end{align*} In view of the bound (\ref{small_z_bound}), both integrals
are absolutely convergent.  Put $w = \frac{4\pi}{c}\sqrt{\ell x}$ and
exchange the order of the integration to rewrite this as 
\begin{align*} 
 & 2 \left(\frac{c}{4\pi \sqrt{\ell}}\right)^{1-2\sigma -2it} 
\frac{1}{2\pi i}\int_{(3)}\left(\frac{4\ell}{c^2d^2}\right)^s
\frac{\Gamma(s + \sigma + \frac{k}{2} + it)\Gamma(s + \sigma +
\frac{k}{2}-it)}{\Gamma(\sigma + \frac{k}{2} + it)\Gamma(\sigma + \frac{k}{2} -
it)}\\
&\qquad\qquad\qquad\qquad\qquad\qquad\qquad\times
\left[\int_0^\infty J_{k-1}(w)w^{1 - \sigma -2it -2s}
\frac{dw}{w}\right]\hat{H}(s) ds
\end{align*}
The bracketed integral is the Mellin transform of
$J_{k-1}$, given by (\ref{mellin_transform}).  

We now pass the summations over $c$ and $d$ under the integral.  Recall
that the Ramanujan sum evaluates to 
\begin{align*}
  &S(0, a;p) = -1,\\&S(0, a; p^e) = 0, \\ &S(0,p;p) = p-1, \qquad
\qquad\qquad (a,p) = 1,\; e
\geq 1
\\ &S(0,ap, p^2) = -p, \\&S(0, ap, p^{e+1}) = 0 
\end{align*}
Thus the resulting Dirichlet series $\sum_{c,d} \frac{S(0,\ell;c)}{(cd)^{1 +
2\sigma + 2s}}$ collapses to the finite product
\begin{align*}&\prod_{p|\ell}\left(1 - \frac{1}{p^{1 + 2\sigma +
2s}}\right)^{-1}\left(1 + \frac{p-1}{p^{1 + 2\sigma + 2s}} -
\frac{p}{p^{2 + 4\sigma +4s}}\right) \\&\qquad= \prod_{p|\ell} \left(1 +
\frac{1}{p^{2\sigma + 2s}}\right) = \ell^{-\sigma - s} \tau_{s+\sigma}(\ell).
 \end{align*}  Combining these steps we arrive at 
\begin{align*}
(\ref{off_diagonal_main})&= o(1) + 
2i^k\Re\Biggl\{\frac{\zeta(1-2it)(2\pi)^{2\sigma + 2it}}{\ell^{1/2 -
it}}\\& \times
\frac{1}{2\pi
i} \int_{(3)} \tau_{s + \sigma}(\ell) \frac{\Gamma(\sigma +
\frac{k}{2} -it + s)\Gamma(-\sigma + \frac{k}{2} -it -s)}{\Gamma(\sigma +
\frac{k}{2} + it)\Gamma(\sigma + \frac{k}{2} - it)} \hat{H}(s) ds\Biggr\}.
\end{align*}

Repeating these steps, one proves that the main terms coming from
$\tilde{g}_{c,d}$ are (again with error $o(1)$)
\begin{align*}
&2i^k\Re\Biggl\{\frac{\zeta(1-2it)(2\pi)^{2\sigma + 2it}}{\ell^{1/2 -
it}}
\\&\qquad\times\frac{1}{2\pi
i} \int_{(3)} \tau_{-s + \sigma}(\ell) \frac{\Gamma(\sigma +
\frac{k}{2} -it - s)\Gamma(-\sigma + \frac{k}{2} -it +s)}{\Gamma(\sigma +
\frac{k}{2} + it)\Gamma(\sigma + \frac{k}{2} - it)} \hat{H}(s) ds\Biggr\}.
\end{align*}
In this integral we change $s$ to $-s$.  Recall that $\hat{H}(-s) =
-\hat{H}(s)$, so that the combined contribution from the $g_{c,d}$ and
$\tilde{g}_{c,d}$ main terms is equal to 
\[
 \frac{1}{2\pi
i} \left\{ \int_{(3)} - \int_{(-3)}\right\} \left[\tau_{s + \sigma}(\ell)
\frac{\Gamma(\sigma +
\frac{k}{2} -it + s)\Gamma(-\sigma + \frac{k}{2} -it -s)}{\Gamma(\sigma +
\frac{k}{2} + it)\Gamma(\sigma + \frac{k}{2} - it)} \hat{H}(s) ds\right]
\]
Thus the two terms together are just equal to the residue of the integrand at
the pole at 0, that is,
\begin{align*}
 &2i^k\Re\left\{\zeta(1-2it)(2\pi)^{2\sigma + 2it}\frac{\tau_\sigma(\ell)}{
\ell^{1/2 -
it}}\frac{\Gamma(-\sigma + \frac{k}{2} -it)}{\Gamma(\sigma + \frac{k}{2} +
it)}\right\} \\&= 2\Re\left\{\zeta(1 - 2it)
\left(\frac{k}{4\pi}\right)^{-2\sigma
-2it}\frac{\tau_{\sigma}(\ell)}{\ell^{1/2 -it}}\right\} + O\left((1 +
t^2) k^{-1}\right)&. 
\end{align*}

\subsection{The terms containing Bessel integrals}
The term (\ref{K_sum}) is extremely small, since the
$K$-Bessel function is exponentially small for large variable and the support of
$F$ in function
$g_{c,d}$ localizes the variable to be of size at least $k
\sqrt{\frac{n}{\ell}}$. The term (\ref{J_sum}) requires some more care,
and we get cancellation from the changing rate of oscillation of the
$J$-Bessel function  in its transition region.

The integral in the term (\ref{J_sum}) is equal to
\begin{align*}
\int_0^\infty  W_{k, \sigma + it}(d^2
x)J_{k-1}\left(\frac{4\pi}{c} \sqrt{\ell x}\right) F\left(\frac{4\pi}{c}
\sqrt{\ell x}\right)
J_{2it}^+\left(\frac{4\pi}{c}\sqrt{nx}\right) \frac{dx}{x^{1/2 +
\sigma}}
\end{align*}
Substituting $y = \frac{4\pi}{c}\sqrt{\ell x}$ we obtain
\begin{align*}
 2\pi(\ref{J_sum}) &= \frac{(4\pi)^{1 + 2\sigma}}{\ell^{1/2-\sigma}}
\sum_{cd < 10000\sqrt{\ell}}\frac{1}{(cd)^{1 + 2\sigma}} \sum_{n = 1}^\infty
\tau_{it}(n)S(0, \ell-n;c) \\ &\qquad\qquad
\times\int_0^\infty W_{k, \sigma + it}\left(\frac{c^2d^2 y^2}{(4\pi)^2
\ell}\right) J_{k-1}(y)J_{2it}^+\left(y\sqrt{\frac{n}{\ell}}\right)F(y)
y^{-2\sigma}dy.
\end{align*}
Now replace $J_{2it}^+$ with its asymptotic expansion
\begin{align*}J_{2it}^+\left(y\sqrt{\frac{n}{\ell}}\right)&=
-\sqrt{\frac{2\pi}{y} \sqrt{\frac{\ell}{n}}}   \sin\left(y\sqrt{\frac{n}{\ell}}-
\frac{\pi}{4}\right) \left[1 - \frac{P(2it)\ell }{128
y^2 n}\right]
\\& \qquad -\pi \cos\left(y \sqrt{\frac{n}{\ell}} -
\frac{\pi}{4}\right)
\frac{-4t^2 -
\frac{1}{4}}{2y}+
O\left(\frac{(1 +t^6)\ell^{\frac{3}{2}}}{y^3n^{\frac{3}{2}}}\right).
\end{align*}
Using the integral bound in Lemma \ref{average_size_bound}, the error 
contributes
$O(\frac{\ell^{2+\sigma}(1 + t^6)}{k^{5/2+2\sigma -\epsilon}})$.  In the
remaining terms we can integrate by parts several times to truncate the sum over
$n$ at $n < \ell k^{\epsilon}$, with negligible error.  We show only how to
bound the contribution from integrating against the main term 
\begin{equation}\label{main_part_of_bessel}-\sqrt{\frac{2\pi}{y}
\sqrt{\frac{\ell}{n}}} 
\sin\left(y\sqrt{\frac{n}{\ell}}
- \frac{\pi}{4}\right); \end{equation} the rest of the main term can be handled
in
exactly the same way, and it produces an error of smaller size.

We will prove the following lemma.

\begin{lem}\label{B_bound}  We have the bound
\begin{align*}\tag{B}\label{main_integral_bound}
 \int_0^\infty W_{k, \sigma + it}\left(\frac{c^2d^2 y^2}{(4\pi)^2
\ell}\right) J_{k-1}(y)\sin\left(y\sqrt{\frac{n}{\ell}} -
\frac{\pi}{4}\right)F(y)
y^{-1/2-2\sigma}dy&\\ \ll \ell^{1/4-\sigma} k^{-1/2 -
2\sigma+\epsilon}&.
\end{align*}
\end{lem}
Assuming this bound for the moment we find that the contribution to
(\ref{J_sum}) from integration against (\ref{main_part_of_bessel}) is 
\[
 \ll \frac{1}{k^{1/2+2\sigma-\epsilon}} \sum_{cd < 10000
\sqrt{\ell}}\frac{1}{(cd)^{1 + 2\sigma}} \sum_{n \ll
\ell k^\epsilon} \frac{|S(0, \ell-n;c)|}{n^{1/4- \epsilon}}
\]
Here the $n = \ell$ term contributes $\ll \ell^{1/4-\sigma}
k^{-1/2-2\sigma + \epsilon}$  while the $n \neq \ell$ terms give
\[ \ll \frac{1}{k^{1/2+2\sigma-\epsilon}}  \sum_{\substack{n \ll \ell
k^\epsilon\\ n \neq \ell}} \frac{1}{n^{1/4-\epsilon}}
\sum_{\substack{c_1 | n-\ell\\ c_1 c_2 d \leq 10000 \sqrt{\ell}}}
\frac{1}{c_1^{2\sigma}c_2^{1 + 2\sigma}d^{1 + 2\sigma}} \ll
\ell^{3/4}k^{-1/2 -2\sigma + \epsilon},\] and both of these bounds
suffice for the theorem.  
The term corresponding to (\ref{J_sum}) coming from $\tilde{g}$ is handled in
an analogous way, so it only remains to prove the bound
(\ref{main_integral_bound}).

\begin{proof}[Proof of Lemma \ref{B_bound}]
We split the integral into the
ranges $y < k - k^{1/3}$, $k-k^{1/3}
< y < k+ k^{1/3}$, and  $k + k^{1/3}<y< 2000k\sqrt{\ell}$.  

For $y < k- k^{1/3}$ we set $y = k - k^\Delta$ so that $w =\sqrt{\frac{k^2}{x^2}
-1}$ satisfies $w \asymp k^{(\Delta -1)/2}$.  Then the bound from
(\ref{just_below}) 
\[J_{k-1}(y) \ll \frac{e^{kw -k \tanh^{-1}w}}{\sqrt{kw}} + O(k^{-4/3})\] easily
suffices for the result, since for small $w$,  \[kw - k \tanh^{-1}w \sim
-\frac{kw^3}{3} \asymp -k^{(3 \Delta - 1)/2}.\] 

For $k-k^{1/3}<y<k + k^{1/3}$ we
bound
simply $J_{k-1}(y)\ll k^{-1/3}$, so that this part also contributes $\ll
k^{-1/2-2\sigma}$.

In the remaining part of
the integral we have from (\ref{just_above})  (set $k' = k-1$), 
\[J_{k'}(y) = \sqrt{\frac{2}{\pi k' w}} \cos\left(k'w -k' \tan^{-1}w -
\frac{\pi}{4}\right) + O\left(k^{-4/3} + \frac{1 + w^{-2}}{k^{3/2}w^{3/2}}
\right)\] with $w =
\sqrt{\frac{y^2}{{k'}^2}-1}.$ For $y > 2k$ we have $w \gg 1$, while for $k
+k^{1/3} < y < 2k$ we have $w \asymp \left(\frac{y-k}{k}\right)^{1/2}$. 
Therefore, integration of the error term
produces
\[ \ll \ell^{1/4 - \sigma} k^{-5/6 -2\sigma}  + \int_{k +
k^{1/3}}^{2k} \frac{k^{1/4}}{y^{1/2 + 2\sigma}(y-k)^{7/4}} \ll \ell^{1/4
- \sigma} k^{-5/6 -2\sigma} + k^{-1/2 -2\sigma}.\]  

Now consider a diadic interval $[k+A, k+2A]$ with $A > k^{1/3}$.  On
such an interval 
we have that $w$ is
fixed to within a constant.  Moreover,\[\sqrt{\frac{2}{\pi k'w}}\cos\left(k'w
-k' \tan^{-1}w -
\frac{\pi}{4}\right)\sin\left(y\sqrt{\frac{n}{\ell}} -
\frac{\pi}{4}\right)\] may be written as a linear combination of exponentials of
the form \[\sqrt{\frac{2}{\pi k' w}}e^{i\left[\pm\left(k'w -k' \tan^{-1}w -
\frac{\pi}{4}\right) \pm\left(y\sqrt{\frac{n}{\ell}} -
\frac{\pi}{4}\right)\right]} = G(y)e^{iF(y)}.\] By further subdividing $[k+A,
k+2A]$ into $O(1)$ subintervals we may assume that $\frac{F'(y)}{G(y)}$ is
monotonic. Recalling (\ref{phase_change}),
\[\frac{d^2}{dy^2} (k'w - k'\tan^{-1}w) = \frac{k'}{y^2 w},\] we obtain from
 Lemma \ref{oscillatory_bound} that for each $B \in [k + A, k+2A]$, 
\begin{align*} \int_{k+A}^B  \sqrt{\frac{2}{\pi k'w}}\cos\left(k'w
-k' \tan^{-1}w -
\frac{\pi}{4}\right)&\sin\left(y\sqrt{\frac{n}{\ell}} -
\frac{\pi}{4}\right)dy \\ &\ll \frac{1}{\sqrt{kw}} \sqrt{\frac{B^2 w}{k}} \ll 1
+
\frac{A}{k}.\end{align*}  Thus summing diadically we
conclude that for all $z \in [k + k^{1/3}, 2000k \sqrt{\ell}]$ we have that
\[I_z = \int_{k + k^{1/3}}^z\sqrt{\frac{2}{\pi k'w}}\cos\left(k'w
-k' \tan^{-1}w -
\frac{\pi}{4}\right)\sin\left(y\sqrt{\frac{n}{\ell}} -
\frac{\pi}{4}\right)dy\] is bounded by $\ll \frac{z \log \ell}{k}.$
Write
\begin{align*}\int_{k + k^{1/3}}^{2000k\sqrt{\ell}} &W_{k, \sigma +
it}\left(\frac{c^2d^2 y^2}{(4\pi)^2
\ell}\right) \sqrt{\frac{2}{\pi k'w}}\cos\left(k'w
-k' \tan^{-1}w -
\frac{\pi}{4}\right)\\& \qquad\qquad\qquad\qquad\qquad
\times\sin\left(y\sqrt{\frac{n}{\ell}} -
\frac{\pi}{4}\right)F(y)\frac{dy}{y^{1/2 + 2\sigma}} \\ & = \int_{k +
k^{1/3}}^{2000k\sqrt{\ell}}W_{k, \sigma +
it}\left(\frac{c^2d^2 y^2}{(4\pi)^2
\ell}\right) F(y)
y^{-1/2-2\sigma}dI_y\end{align*} and integrate by parts.  Substituting
our absolute bound for $I_y$ and the bounds
\[\frac{\partial}{\partial y} W_{k, \sigma +
it}\left(\frac{c^2d^2 y^2}{(4\pi)^2
\ell}\right) \ll \frac{1}{y}, \qquad \qquad F'(y) \ll \frac{1}{y}\]
gives the result $\ll
k^{-1/2-2\sigma}\ell^{1/4 -\sigma}\log
\ell.$  This completes the bound (\ref{main_integral_bound}).
\end{proof}

\section{Mollification}\label{mollifier_section}
Write the inverse of the $L$ function $L(s;f)$
as 
\[L(s;f)^{-1} = \sum_{n = 1}^\infty \frac{a_f(n)}{n^s} = \prod_p \left(1 -
\frac{\lambda_f(p)}{p^s} + \frac{1}{p^{2s}}\right), \qquad \qquad \Re(s) > 1.\]
The coefficients $a_f(n)$ are supported on cube-free numbers, and for $m,n$
square-free, $(m,n) = 1$ we have $a_f(mn^2) = \mu(m)\lambda_f(m)$.  We define a
mollifier for $L(s;f)$ by
\begin{equation}\label{mollifier}M\left(s;f\right) = \sum_{n =
1}^\infty
\frac{a_f(n)F(s(n))}{n^{s}}.\end{equation}  Here $s(n) = \prod_{p|n} p$ denotes
the
squarefree
kernel of $n$ and $F(n)$ is a cut-off function to be given explicitly later,
but for which we stipulate $F(n) \ll n^\epsilon$ and $F(n) = 0$ for $n > M =
k^\theta$ for some $\theta < \frac{1}{5}$.
In particular, we have the representation
\begin{align} &\notag\left|M\left(1/2 +\sigma + it;f\right)\right|^2 =\left|
{\sum_{(m, n)
= 1}}^\flat \frac{\mu(m) \lambda_f(m)
F(mn)}{m^{1/2 + \sigma + it}n^{1 + 2\sigma +2it}}\right|^2 \\
\label{mollifier_squared}&\qquad=
{\sum_{d}}^\flat \frac{1}{d^{1 + 2\sigma}} {\sum_{\substack{(m_1,n_1)
= 1\\ (m_2,n_2) = 1\\ (m_1n_1m_2n_2,d) =1 }}}^\flat
\frac{\mu(m_1)\mu(m_2)\lambda_f(m_1m_2)F(dm_1n_1)F(dm_2n_2)}{m_1^{1/2 +
\sigma + it}m_2^{1/2 + \sigma - it}n_1^{1 + 2\sigma + 2it}n_2^{1 +
2\sigma - 2it}}
\end{align}
From this representation, we find
\begin{align}\notag  & {\sum_{f \in H_k}}^h
\left|M L \left(\frac{1}{2} + \sigma +
it;f\right)\right|^2 \\&\notag \qquad= 
 {\sum_d}^\flat \frac{1}{d^{1 + 2\sigma}}{\sum_{\substack{(m_1,n_1) = 1\\
(m_2,n_2) = 1\\
(m_1n_1m_2n_2, d)=1}}}^\flat \frac{\mu(m_1)\mu(m_2)
F(m_1n_1d)F(m_2n_2d)}{m_1^{1/2 + \sigma + it}m_2^{1/2 + \sigma
- it}n_1^{1 + 2\sigma + 2it}n_2^{1 + 2\sigma - 2it}} \\&
\label{mollified_moment}
\qquad\qquad\qquad\qquad\qquad\times
{\sum_{f \in H_k}}^h \lambda_f(m_1m_2)\left|L\left(\frac{1}{2} + \sigma +
it;f\right)\right|^2
\end{align}
Substituting our expression for the twisted second moment, we find that
\begin{align*} \text{expr. }&(\ref{mollified_moment}) + O(k^{5\theta/2 -
2\sigma -1/2 + \epsilon}) \\ &= {\sum_d}^\flat
\frac{1}{d^{1 + 2\sigma}}{\sum_{\substack{(m_1,n_1) = 1\\
(m_2,n_2) = 1\\
(m_1n_1m_2n_2, d)=1}}}^\flat \frac{\mu(m_1)\mu(m_2)
F(m_1n_1d)F(m_2n_2d)}{m_1^{1/2 + \sigma + it}m_2^{1/2 + \sigma
- it}n_1^{1 + 2\sigma + 2it}n_2^{1 + 2\sigma - 2it}}\\
& \qquad \times \Biggl\{ 
\zeta(1 +
2\sigma) \frac{\tau_{it}(m_1m_2)}{(m_1m_2)^{1/2 + \sigma}}
 +  
\zeta(1-2\sigma) 
\left(\frac{k}{4\pi}\right)^{-4\sigma}\frac{\tau_{it}(m_1m_2)}{(m_1m_2)^{1/2 - \sigma}}
\\&
\qquad \qquad\qquad\qquad+ i^k2 \Re \left[ \zeta(1 +
2it) \left(\frac{k}{4\pi}\right)^{ - 2\sigma +
2it}\frac{\tau_{\sigma}(m_1m_2)}{(m_1m_2)^{1/2 +
it}} \right] 
\Biggr\}\\&= \cS_1 + \cS_2 +2i^k \Re \cS_3 
\end{align*}
We may rewrite the divisor sums 
\[\tau_s(m_1m_2) = \sum_{\ell_1\ell_2 =
m_1m_2}\left(\frac{\ell_1}{\ell_2}\right)^s =
\sum_{g|(m_1,m_2)}\mu(g)
\tau_s\left(\frac{m_1}{g}\right)\tau_s\left(\frac{m_2}{g}\right). \] Doing so
and shifting the sum
over $g$ to the front we separate the variables $m_1$ and $m_2$.  Thus we find:
\[\cS_1 = \zeta(1 + 2\sigma) {\sum_{d}}^\flat
\frac{1}{d^{1 + 2\sigma}} \sum_{(g,d) =1} \frac{\mu(g)}{g^{2 +
4\sigma}}\left|{\sum_{\substack{(m,n)=1\\ (mn,gd)=1}}}^\flat
\frac{\mu(m)\tau_{it}(m)
F(mngd)}{m^{1 + 2\sigma + it}n^{1 + 2\sigma + 2it}}\right|^2\] and similar
expressions for
$\cS_2$, and $\cS_3$, although the inner sum in $\cS_3$ is not a square. In
fact, there is substantial cancellation in the inner summation
for
$\cS_1$ above coming from the M\"{o}bius function.  The sum is in fact equal
to
\[\cS_1 = \zeta(1 + 2\sigma) {\sum_{d}}^\flat
\frac{1}{d^{1 + 2\sigma}} \sum_{(g,d) =1} \frac{\mu(g)}{g^{2 +
4\sigma}}\left|{\sum_{\substack{(m,gd)=1}}}^\flat \frac{\mu(m)
F(mgd)}{m^{1 + 2\sigma }}\right|^2\]
We also find
\begin{align*}\cS_2 = \zeta(1 - 2\sigma)
\left(\frac{k}{4\pi}\right)^{-4\sigma}{\sum_{d}}^\flat
&\frac{1}{d^{1 + 2\sigma}} \sum_{(g,d) =1}
\frac{\mu(g)}{g^2}\\&\qquad\times\left|{\sum_{\substack{(m,n)=1\\
(mn,gd)=1}}}^\flat
\frac{\mu(m)\tau_{it}(m)
F(mngd)}{m^{1+ it}n^{1 + 2\sigma + 2it}}\right|^2,\end{align*}

\begin{align*}\cS_3 =& \zeta(1 + 2it)
\left(\frac{k}{4\pi}\right)^{-2\sigma + 2it} {\sum_{d}}^\flat
\frac{1}{d^{1 + 2\sigma}} \sum_{(g,d) =1} \frac{\mu(g)}{g^{2 +
2\sigma+2it}}\\ &
\times{\sum_{\substack{(m_1,gd)=1}}}^\flat \frac{\mu(m_1)
F(m_1gd)}{m_1^{1 + 2it}}{\sum_{\substack{(m_2, gd) = 1 \\ m_2^1 m_2^2
m_2^3 =
m_2 }}}^\flat \frac{\mu(m_2^1)\mu(m_2^2)F(m_2gd)}{(m_2^1)(m_2^2)^{1 +
2\sigma}(m_2^3)^{1 +
2\sigma - 2it}}.
\end{align*}

\subsection{Upper bound for the harmonic mollified second moment}
We now fix the cut-off function $F$ and prove an upper bound for the mollified
second moment.  Let 
\begin{equation}\label{cutoff_function}
 F(x) = \left\{\begin{array}{lll} 1 &&0 \leq x \leq \sqrt{M}\\
P\left(\frac{\log(\frac{M}{x})}{\log M}\right) && \sqrt{M} \leq x \leq M\\
0 && x \geq M \end{array}\right.
\end{equation}
where $P(t) = 12t^2 -16t^3$ satisfies $P(\frac{1}{2}) = 1$ and $P'(\frac{1}{2})
= P(0) = P'(0) = 0$.  The function $F$ is continuously differentiable.  It's
Mellin transform is equal to 
\begin{equation}\label{mellin_F}
 \hat{F}(s) = \frac{24(M^s + M^{\frac{s}{2}})}{s^3(\log M)^2} - \frac{96(M^s -
M^{\frac{s}{2}})}{s^4 (\log M)^3}.
\end{equation}
It has a simple pole at $s = 0$ with residue 1.  Also, expanding $\hat{F}(s)$
in it's Laurent series about 0,
\begin{equation}\label{laurent}\hat{F}(s) = \frac{1}{s} + \sum_{n = 0}^\infty
c_n s^n\end{equation}
the coefficients $c_n$ satisfy the bound
\[c_n \ll \frac{(\log M)^{n+1}}{(n+3)!}.\]

For this choice of cut-off function we prove
\begin{prp}\label{mollifier_bound}
 Let $M = k^\theta$ with $\theta < \frac{1}{5}$ and
suppose $
\frac{1}{\log k} < \sigma $ and $|t|<k^{1/4}$. For $M(\frac{1}{2} +
\sigma + it;f)$ defined by
(\ref{mollifier}) and cut-off function $F$ as in (\ref{cutoff_function}) we have
\[ {\sum_{f \in H_k}}^h \left|M L \left(\frac{1}{2} + \sigma +
it;f\right)\right|^2 \leq 1+
O(k^{5\theta/2 - 2\sigma - 1/2 + \epsilon}) + O(k^{- \theta
\sigma}).
\]
\end{prp}

\begin{proof}
 We prove $\cS_1 = 1 + O(K^{- \theta \sigma})$ and
$\Re(\cS_3) = O(K^{ -\theta \sigma})$.  This suffices because $\cS_2 \leq
0$ since $\zeta(1-2\sigma) < 0$.

By Mellin inversion
\begin{equation}\label{S1_integral}\cS_1 = \zeta(1 + 2\sigma)\left(\frac{1}{2\pi
i}\right)^{2}\int_{(2)}\int_{(2)}\hat{F}(\alpha)\hat{F}(\beta)G(\alpha, \beta;
\sigma)d\alpha d\beta
\end{equation}
where
\begin{align*}
 G(\alpha, \beta; \sigma) &=
{\sum_{d}}^\flat
\frac{1}{d^{1 + 2\sigma+\alpha + \beta}} \sum_{(g,d) =1} \frac{\mu(g)}{g^{2 +
4\sigma+\alpha + \beta}}{\sum_{\substack{(m_1,gd)=1\\(m_2, gd) = 1}}}^\flat
\frac{\mu(m_1)\mu(m_2)
}{m_1^{1 + 2\sigma + \alpha}m_2^{1 + 2\sigma + \beta}}
\\&= \prod_{p}\left(1 - p^{-1-2\sigma - \alpha} -
p^{-1-2\sigma - \beta} + p^{-1 - 2\sigma - \alpha -\beta}\right)\\
&= \frac{\zeta(1 + 2\sigma + \alpha + \beta)}{\zeta(1 + 2\sigma +
\alpha)\zeta(1 + 2\sigma + \beta)}H(\alpha, \beta; \sigma)
\end{align*}
The Euler product defining $H$ converges absolutely in the region \[\alpha +
2\sigma > -\frac{1}{2},\quad \beta + 2\sigma > -\frac{1}{2},  \quad\alpha +
\beta +
2\sigma > - \frac{1}{2}.\]

To evaluate the integral, shift both contours to the line $\Re(\alpha) =
\Re(\beta) = \frac{1}{\log k}$ and truncate the $\beta$ integral at
$|\Im(\beta)| \leq k$ with error $O(k^{-2 + \epsilon})$.  Then shift the
$\alpha$ integral to the contour $\mathcal{C}$ given by
\[\mathcal{C} := \{\alpha: \Re(\alpha) = -2\sigma - \log
^{3/4}(2 + |\Im(\alpha)|)\}.\] 

In shifting the $\alpha$ contour to $\mathcal{C}$ we encounter poles at $\alpha
= 0$ and $\alpha = -2\sigma - \beta$.  This first pole yields a residue
\begin{equation}\label{first_pole_S1}\frac{1}{\zeta(1 + 2\sigma)} \frac{1}{2\pi
i}\int_{\frac{1}{\log k} -
ik}^{\frac{1}{\log k} + i k}\hat{F}(\beta)
d\beta = \frac{1 + O(k^{-2})}{\zeta(1 + 2\sigma)}.\end{equation}

The second pole has residue
\[ \frac{1}{2\pi i} \int_{1 + \frac{1}{\log k} -
ik}^{1 + 1/\log k + i k} \hat{F}(\beta)
\hat{F}(-2\sigma - \beta) \frac{H(-2\sigma - \beta, \beta;
\sigma)}{\zeta(1-\beta)\zeta(1+2\sigma + \beta)}d\beta\]
Here we can extend the integration to the full line, and shift the contour to
$\Re(\beta) = -\sigma$.  On this line, $H(-\sigma + is, -\sigma - is;\sigma)$
is uniformly bounded, and so the integral is bounded by
\begin{align}\label{second_res_bound} \int_{-\infty}^\infty
\left|\frac{\hat{F}(-\sigma + is)}{\zeta(1+\sigma -is)}\right|^2 ds
 &\ll M^{-\sigma}\Biggl[\int_{-1}^1 \left|\frac{\log^{-2}M}{|\sigma +
 is|^2 } + \frac{\log^{-3} M}{|\sigma + is|^3 }\right|^2 d|s|\\&\notag
\qquad\qquad\qquad\qquad\qquad + O\left((\log M)^{-4}\right)\Biggr].\end{align}
Now using $(a + b)^2 \leq 2(a^2 + b^2)$, the right hand side is bounded by 
\[
k^{-\theta \sigma} \left[O\left((\log k)^{-4}\right) +  \frac{1}{(\log
M)^4}\int_{-\infty}^\infty \frac{ds}{(s^2 + \sigma^2)^2} + \frac{1}{(\log M)^6}
\int_{-\infty}^\infty \frac{ds}{(s^2 + \sigma^2)^3}\right] 
\]
Since $\sigma \geq \frac{1}{\log k}$ we deduce that
the second residue is 
$\ll \frac{k^{-\theta \sigma}}{\log
k}.$
The remaining integral, for $\alpha$ on $\mathcal{C}$, is bounded using standard
bounds for $\zeta$ in the zero-free region and is quite small.  Since
with $\zeta(1 + 2\sigma) \ll \log k$ we have the claimed evaluation of $\cS_1$.

In bounding $2 \Re(\cS_3)$ we handle separately the cases
$t \leq \frac{1}{4\log k}$ and $t > \frac{1}{4\log k}$.

When $t > \frac{1}{4\log k}$ we bound $\cS_3$ in magnitude as we did $\cS_1$.
By Mellin inversion
\begin{equation}\label{int_for_S3}\cS_3 = \zeta(1 +
2it)\left(\frac{k}{4\pi}\right)^{ - 2\sigma + 2it} \left(\frac{1}{2\pi
i}\right)^2 \int_{(2)}\int_{(2)} \hat{F}(\alpha)\hat{F}(\beta) G(\alpha, \beta;
\sigma, t)d\alpha d\beta\end{equation}
where now $G(\alpha, \beta; \sigma,t)$ is given by
\begin{align*}
 G(\alpha, \beta; \sigma, t) &=
{\sum_{d}}^\flat
\frac{1}{d^{1 + 2\sigma+\alpha +\beta}} \sum_{(g,d) =1} \frac{\mu(g)}{g^{2 +
2\sigma+2it+\alpha + \beta}}{\sum_{\substack{(m_1m_2,gd)=1}}}^\flat
\frac{\mu(m_1)
}{m_1^{1 + 2it+\alpha}}
\\& \qquad \times\sum_{m_2^1 m_2^2 m_2^3 =
m_2} \frac{\mu(m_2^1)\mu(m_2^2)}{(m_2^{1})^{1 + \beta}(m_2^2)^{1 +
2\sigma + \beta}(m_2^3)^{1 +
2\sigma - 2it + \beta}}
\\&= \prod_p \Biggl[1- \frac{1}{p^{1 + \beta}} 
- \frac{1}{p^{1 + 2\sigma + \beta}} + \frac{1}{p^{1 + 2\sigma - 2it +
\beta}} - \frac{1}{p^{1 + 2it + \alpha}} \\&
\qquad\qquad\qquad\qquad + \frac{1}{p^{1 +
2\sigma +
\alpha + \beta}}+ \frac{1}{p^{2 +
2it + \alpha + \beta}} - \frac{1}{p^{2 + 2\sigma +\alpha + \beta}}\Biggr]
\\ &= \frac{\zeta(1 + 2\sigma -2it + \beta) \zeta(1 + 2\sigma + \alpha +
\beta)}{\zeta(1 + \beta)\zeta(1 + 2\sigma + \beta) \zeta(1 + 2it +
\alpha)}H(\alpha, \beta; \sigma, t)
\end{align*}
Here the Euler product defining $H(\alpha,
\beta;\sigma, t)$ converges absolutely for \[\min(\Re(\alpha),
\Re(\beta), \Re(\alpha + \beta)) > -\frac{1}{2}.\]

To evaluate the integral in (\ref{int_for_S3}), shift $\alpha$ and $\beta$
contours to the lines $\Re(\alpha)= \Re(\beta) = \frac{1}{\log k}$.  We may
assume that  $\sigma < \frac{100 \log \log k}{\log k}$ since otherwise
$k^{-2\sigma} < \frac{k^{-\sigma}}{(\log k)^{100}}$ and the integral
may be bounded directly using standard bounds for $\zeta$ and $\zeta^{-1}$ to
the right of the 1-line. Now
truncate the $\beta$ contour at $|\Im(\beta)| < k$ and shift the
$\alpha$ contour to $\mathcal{C}'$ given by
\[\mathcal{C}':= \{\alpha: \Re(\alpha) = - \log^{3/4}(2 + |\Im(\alpha +
2it)|)\}.\] In doing so we pass two poles, at $\alpha = 0$ and at
$\alpha = -\beta - 2\sigma$.  The first pole has residue
\begin{align*}&\zeta(1 + 2it)^{-1} \frac{1}{2\pi i} \int_{\frac{1}{\log k} -
ik}^{\frac{1}{\log k} + ik} \frac{\hat{F}(\beta)\zeta(1 +
2\sigma -2it + \beta)}{\zeta(1+\beta)}H(0,\beta; \sigma, t)d\beta \\&= \zeta(1 +
2it)^{-1} \left(\frac{\hat{F}(-2\sigma + 2it)}{\zeta(1 -2\sigma + 2it)}H(0,
-2\sigma + 2it; \sigma ,t)+ O(1)\right). \end{align*}  Expressing
$\hat{F}(-2\sigma + 2it)$ using either the Laurent expansion for (\ref{laurent})
for $|t| < \frac{1}{\log k}$ or the direct definition (\ref{cutoff_function})
for $|t| > \frac{1}{\log k}$, together with the bound $\frac{1}{\zeta(1 -s)}
\ll s$ valid in the standard zero-free region  we have that this residue is
$O(\zeta(1
+ 2it)^{-1})$.

The second residue is equal to 
\[\frac{1}{2\pi i} \int_{\frac{1}{\log k}- ik}^{\frac{1}{\log k} +
ik} \frac{\hat{F}(-\beta - 2\sigma) \hat{F}(\beta)\zeta(1 + 2\sigma
-2it + \beta)}{\zeta(1 + \beta) \zeta(1 + 2\sigma + \beta) \zeta(1 + 2it -
2\sigma - \beta)}H(-2\sigma - \beta,\beta)d\beta\]
Shifting this integral to the line $\Re(\beta) = -2\sigma$ (the horizontal
integrals are very small), and taking absolute values, we obtain a bound
\begin{align*} \ll \int_{-k}^{k} \frac{|\hat{F}(-2\sigma
+is)|}{|\zeta(1-2\sigma +is)|}\frac{|\hat{F}(-is)|}{|\zeta(1 -is)|} ds.
\end{align*}
Arguing as above we have for all real $s$, 
$ \frac{|\hat{F}(-is)|}{|\zeta(1
-is)|} = O(1)$ while for $|s|\leq k$,
\[
\frac{|\hat{F}(-2\sigma
+is)|}{|\zeta(1-2\sigma +is)|} \ll M^{-\sigma} \left[\frac{1}{(\log M)^2
|\sigma + is|^2} + \frac{1}{(\log M)^3 |\sigma + is|^3}\right].
\]
so that the integral is $ O(\frac{k^{-\theta \sigma}}{\log
k})$ as in the second residue calculation for $\cS_1$.  
The remaining double integral with $\alpha$ on the contour $\mathcal{C} - 2it$
is again small.   Thus for 
$\frac{1}{4\log k} <
t$,  we have \[O(\zeta(1 +
it)^{-1}) + O\left(\frac{k^{-\theta \sigma}}{\log
k}\right)\] for the integral in (\ref{int_for_S3}), which suffices since in this
range, $\zeta(1 + 2it) = O(\log k)$.

When $|t| < \frac{1}{4\log k}$, we bound $2\Re\cS_3$ to balance
the
fact that $\zeta(1 + 2it)$ can be quite large (but mostly imaginary).
Following the method of \cite{conrey_sound}, let $\cO$ be the circle $|w| =
\frac{1}{2\log k}$. By Cauchy's residue theorem
\begin{align*}
2\Re\cS_3  = \frac{1}{2\pi i} \int_{\cO}
\left(\frac{k}{4\pi}\right)^{ - 2\sigma +w}\zeta(1 + w)  \eta(w; \sigma)
\left[\frac{1}{w + 2it} + \frac{1}{w - 2it} \right]dw
\end{align*}
with
\begin{align*}
 \eta(w;\sigma)&= {\sum_{d}\frac{1}{d^{1 +
2\sigma}}}^\flat   \sum_{(g,d) = 1}
\frac{\mu(g)}{g^{2 + 2\sigma + w}}{\sum_{(m_1m_2, gd) = 1}}^\flat
\frac{\mu(m_1)F(m_1gd)F(m_2gd)}{m_1^{1 + w}}\\& \qquad\qquad\qquad\qquad
\times  \sum_{m_2^1m_2^2m_2^3 = m_2}
\frac{\mu(m_2^1)\mu(m_2^2)}{(m_2^1)(m_2^2)^{1 + 2\sigma}(m_2^3)^{1 + 2\sigma -
w}}
\end{align*}
As before,  we may assume that $\sigma <
\frac{100 \log \log k}{\log k}$. The evaluation
of $\eta(w;\sigma)$ by
Mellin inversion is exactly analogous to the integral performed in calculating
$\cS_3$ when $\frac{1}{4\log k}< |t|$: there is a main
term equal to $\zeta(1+w)^{-1}O(1)$, a secondary residue term of size
$\frac{M^{-\sigma}}{\log k}$ and a smaller error integral. Thus $\eta(w;\sigma)
= O(\frac{1}{\log k})$. Thus the integrand in the integral over $\cO$ is
$O(k^{
- \theta\sigma}\log k)$. Since the length of the integral is $O(\frac{1}{\log
k})$
the integral itself is $O(k^{ -\theta\sigma})$. 
\end{proof}

\section{Removing the harmonic weights}
The starting point for the Kowalski-Michel \cite{kowalski_michel} method  for
removing harmonic weights is the  formula (\cite{iwaniec_luo_sarnak})
\[ w_f^{-1} = \frac{L(1, \sym^2 f)}{\zeta(2)} |H_k|  +
O(\log^3 k)\]
where $L(s, \sym^2 f)$ is the symmetric square $L$-function associated to $f$,
defined by
\[L(s, \sym^2 f) = \sum_{n = 1}^\infty \frac{\rho_f(n)}{n^s} = \zeta(2s)
\sum_n \frac{\lambda_f(n^2)}{n^s}.\]
Thus the natural average is expressed as 
\begin{align*}&\frac{1}{|H_k|} \sum_{f \in H_k} \left|ML\left(1/2
+
\sigma +it; f\right)\right|^2 = \frac{1}{|H_k|} {\sum_{f \in H_k}}^h w_f^{-1}
\left|ML\left(1/2 + \sigma +it;
f\right)\right|^2\\&\qquad\qquad\qquad= \frac{1}{\zeta(2)}
{\sum_{f \in H_k}}^h L(1, \sym^2 f) \left|ML(1/2 + \sigma +it;
f)\right|^2 + O(k^{-1+\epsilon}).\end{align*}  The method replaces $L(1, \sym^2
f)$ with a short Dirichlet polynomial approximation
\[w_f(x) = \sum_{n \leq x} \frac{\rho_f(n)}{n}, \qquad \qquad x = k^{\kappa}.\] 
A minor modification to the proof of Proposition 2 of \cite{kowalski_michel}
yields the following result.
\begin{prp}
 Assume that the mollifier $M(\frac{1}{2} + \sigma + it; f)$ is such that
\begin{equation}\label{sup_bound}\sup_{f \in H_k} w_f \left|ML\left(\frac{1}{2}
+ \sigma + it; f\right)\right|^2 < k^{-\delta}, \qquad \qquad
\delta >
0\end{equation}
and
\begin{equation}\label{L2_bound}{\sum_{f \in H_k}}^h \left|ML\left(\frac{1}{2}
+ \sigma + it; f\right)\right|^2 < (\log k)^A.\end{equation}
Let $x = k^\kappa$ for some $\kappa > 0$.  Then there is a $\gamma =
\gamma(\delta, \kappa, A) > 0$ such that
\begin{align*} &\frac{1}{|H_k|} \sum_{f \in H_k} \left| ML\left(\frac{1}{2} +
\sigma +
it;
f\right)\right|^2 \\&\qquad= {\sum_{f \in H_k}}^h w_f(x)\left|
ML\left(\frac{1}{2}
+ \sigma +
it;
f\right)\right|^2 + O(k^{-\gamma}).\end{align*}
\end{prp}
The result of the previous section guarantees condition (\ref{L2_bound}) so long
as the mollifier has length $M = k^\theta$, $\theta < \frac{1}{5}$.  
Trivially $|M(\frac{1}{2} + \sigma + it)| < k^{\theta/2 + \epsilon}$ 
and the best known subconvex bound (see \cite{jutila_motohashi}) gives
$L(\frac{1}{2} + \sigma + it)\ll (k + |t|)^{1/3 - 2\sigma/3 +
\epsilon}$.  Thus condition (\ref{sup_bound}) holds uniformly in $|t|< k$ for
$\theta < \frac{1}{3}$. Therefore, we  complete the proof of Proposition
\ref{natural_mollification} by proving the following uniform bound.

\begin{prp}\label{truncated}
 For sufficiently small $\kappa, \delta, \theta > 0$ there exists $
\gamma(\kappa, \delta, \theta) > 0$ such that, uniformly in
$\frac{1}{\log k} < \sigma \leq 1$ and $|t| < k^\delta$,
\[\frac{1}{\zeta(2)} {\sum_{f \in H_k}}^h \left(\sum_{n \leq x= k^\kappa}
\frac{\rho_f(n)}{n}\right)\left|M L(\frac{1}{2} + \sigma + it;f)\right|^2
\leq 1 + O(k^{-\theta \sigma} + k^{-\kappa/2 + \epsilon}),\]
where $M$ is the mollifier from the previous section, having length $M =
k^\theta$.
\end{prp}
\subsection{Proof of Proposition \ref{truncated}}
Combining expression (\ref{mollifier_squared}) for $|M(\frac{1}{2} + \sigma +
it;f)|^2$ with $\sum_{n \leq x} \frac{\rho_f(n)}{n} = \sum_{\ell^2 d < x}
\frac{\lambda_f(d^2)}{\ell^2 d}$ and the Hecke relations, we obtain
\begin{align*} &\sum_{n \leq x} \frac{\rho_f(n)}{n} \left|M\left(\frac{1}{2} +
\sigma +
it; f\right)\right|^2 = \sum_{\ell^2 d < x} \frac{1}{\ell^2 d}{\sum_g}^\flat
\frac{1}{g^{1 + 2\sigma}}\\& \times{\sum_{\substack{(m_1, n_1) = 1\\
(m_2, n_2) =
1\\ (m_i n_i, g) = 1}}}^\flat
\frac{\mu(m_1)\mu(m_2)F(m_1n_1g)F(m_2n_2g)}{m_1^{1/2 + \sigma +
it}m_2^{1/2 + \sigma - it}n_1^{1 + 2\sigma + 2it}n_2^{1 + 2\sigma -
2it}}  \sum_{h|(d^2, m_1m_2)}
\lambda_f\left(\frac{m_1m_2d^2}{h^2}\right). \end{align*}  Write $h = h_1
h_2^2$ where $h_1$ and $h_2$ are squarefree.  Clearly $h_2 |(m_1,m_2)$ and 
$h_2|d$.  Also, $h_1 | (d, m_1m_2)$.  Shifting the orders of summation, and
then introducing our expression for the twisted second moment, we obtain
\begin{align*}& {\sum_{f \in H_k}}^h \left(\sum_{n \leq x= k^\kappa}
\frac{\rho_f(n)}{n}\right)\left|M L\left(\frac{1}{2} + \sigma +
it;f\right)\right|^2=\sum_\ell \frac{1}{\ell^2}{\sum_{(g, h_2)=1}}^\flat
\frac{1}{g^{1 +
2\sigma }h_2^{2 + 2\sigma}}
\\&\quad \times
{\sum_{\substack{(m_1,n_1) = 1\\ (m_2,n_2) = 1\\ (m_1m_2n_1n_2, gh_2)=1}}}^\flat
\frac{\mu(m_1)\mu(m_2)F(m_1n_1gh_2)F(m_2n_2gh_2)}{m_1^{1/2 + \sigma +
it}m_2^{1/2 + \sigma -it}n_1^{1 + 2\sigma + 2it}n_2^{1 + 2\sigma - 2it}}
{\sum_{h_1|m_1m_2}}^\flat \frac{1}{h_1}
\\ &\quad  \times \sum_{d <
\frac{x}{\ell^2 h_1h_2}} \Biggl\{ 
\frac{\zeta(1 + 2\sigma)\tau_{it}(m_1m_2d^2)}{d (m_1m_2d^2)^{1/2 +
\sigma}} +
\left(\frac{k}{4\pi}\right)^{-4\sigma}
\frac{\zeta(1-2\sigma)\tau_{it}(m_1m_2 d^2)}{d(m_1m_2d^2)^{1/2 -
\sigma}} + \\ &  +2i^k\Re\left(
\left(\frac{k}{4\pi}\right)^{-2\sigma +
2it}\frac{\zeta(1 + 2it)\tau_\sigma(m_1m_2d^2)}{d(m_1m_2d^2)^{1/2 +it}} \right)
+ O\left(\frac{d^{3/2}
(m_1m_2)^{3/4}}{k^{1/2 +\sigma - \epsilon}} \right)\Biggr\}\end{align*}
The error term contributes
$\ll k^{-1/2 + 5\theta/2 -2\sigma + \epsilon}x^{2 + 2\sigma} \ll k^{-1/2 +
5\theta/2 + \epsilon} x^2.$

For $\sigma > \frac{1}{4}$, the terms involving $\zeta(1 - 2\sigma)$ and
$\zeta(1 + 2it)$ are negligibly small and so we are left to consider only the
$\zeta(1 + 2\sigma)$ term; otherwise, for $\frac{1}{\log k} < \sigma <
\frac{1}{4}$ we consider all three terms.  In either case, we may remove the
restriction on the sum over $d$ with error $\ll x^{-1/2 + \epsilon}$.
Thus
\[\frac{1}{\zeta(2)}{\sum_{f \in H_k}}^h \left(\sum_{n \leq x= k^\kappa}
\frac{\rho_f(n)}{n}\right)\left|M\cdot L\left(\frac{1}{2} + \sigma +
it;f\right)\right|^2 = \mathcal{S}_1 + \mathcal{S}_2 + 2\Re \mathcal{S}_3\] with
the
stipulation that $\mathcal{S}_2 = \mathcal{S}_3 = 0$ if $\sigma > \frac{1}{4}$. 

We use the following lemma.
\begin{lem}
 Let $m_1$ and $m_2$ be squarefree.  For $\Re(s \pm \gamma) > 1$ we have
\begin{align*}\sum_d \frac{\tau_\gamma(m_1m_2 d^2)}{d^s} &=
\frac{\zeta(s)}{\zeta(2s)}
\zeta(s + 2\gamma) \zeta(s-2\gamma)\\&\quad\times
\prod_{p|\frac{m_1m_2}{(m_1,m_2)^2}}
\frac{p^\gamma + p^{-\gamma}}{1 + p^{-s}} \prod_{p|(m_1,m_2)}\frac{1 +
p^{2\gamma} + p^{-2\gamma} - p^{-s}}{1 + p^{-s}}.\end{align*}
\end{lem}

We first prove that for $\sigma < \frac{1}{4}$,  $\mathcal{S}_2< 0$, so that
it may be completely discarded. Since we assume $\sigma < \frac{1}{4}$,
\begin{align*}\mathcal{S}_2 &= \zeta(1-2\sigma) \frac{\zeta(2 -
2\sigma)}{\zeta(4-4\sigma)} \left|\zeta(2-2\sigma + 2it)\right|^2 {\sum_{k =
gh}}^\flat \frac{1}{g^{1 -
2\sigma}h^{2-2\sigma}}\\&\times{\sum_{\substack{(m_1,n_1)=1\\(m_2,
n_2)=1\\(m_1n_1m_2n_2,k )=1} } }^\flat
\frac{\mu(m_1)\mu(m_2)F(m_1n_1k)F(m_2n_2k)}{m_1^{1 + it}m_2^{1-it}n_1^{1 +
2\sigma + 2it}n_2^{1 + 2\sigma - 2it}} 
 \\& \times \prod_{p|m_1}\left(\frac{p+1}{p} \cdot
\frac{p^{it} + p^{-it}}{1 + p^{-2+2\sigma}}\right)\prod_{p|m_2}\left(
\frac{p+1}{p}\cdot\frac{p^{it} + p^{-it}}{1 +
p^{-2+2\sigma}}\right)
\\& \times \prod_{p|(m_1,m_2)}\left(\frac{p}{p+ 1}\cdot \frac{(1 + p^{2it} +
p^{-2it} -
p^{-2+2\sigma})(1 + p^{-2+2\sigma})}{(p^{it} + p^{-it})^2}\right)\end{align*}
This may be rearranged as
\begin{align*}&\zeta(1-2\sigma) \frac{\zeta(2 -
2\sigma)}{\zeta(4-4\sigma)} \left|\zeta(2-2\sigma + 2it)\right|^2{\sum_{k =
ghr}}^\flat \frac{a(r)}{g^{1 -
2\sigma}h^{2-2\sigma}r^2}\\&
\qquad \times\left|{\sum_{\substack{(m_1,n_1)=1\\(m_1n_1,k
)=1} } }^\flat
\frac{\mu(m_1)F(m_1n_1k)}{m_1^{1 + it}n_1^{1 +
2\sigma + 2it}} \prod_{p|m_1}\left(
\frac{p+1}{p} \cdot \frac{p^{it} + p^{-it}}{1 + p^{-2+2\sigma}}\right)
\right|^2,\end{align*}
where $a(r)$ is the multiplicative function, supported on squarefree integers,
and given on primes by
\begin{align*}a(p) &= \frac{p+1}{p}\cdot \frac{1 + p^{2it} + p^{-2it} - p^{-2 +
2\sigma}}{1 + p^{-2 + 2\sigma}} - \left(\frac{p+1}{p}\frac{p^{it} + p^{-it}}{1 +
p^{-2 + 2\sigma}}\right)^2\\& = -\frac{p+1}{p} - (p^{it} + p^{-it})^2
\left[\left(\frac{p+1}{p + p^{-1 + 2\sigma}} \right)^2 -
\frac{p+ 1}{p + p^{-1 + 2\sigma}}\right].\end{align*}  Now observe
\[\sum_{ghr= k} \frac{a(r)}{g^{1-2\sigma}h^{1-2\sigma}r^2} = \prod_{p|k}b(p);
\qquad \qquad b(p) =  \frac{1}{p^{1-2\sigma}} + \frac{1}{p^{2-2\sigma}} +
\frac{a(p)}{p^2}.\]  We have $b(p) \geq 0$; indeed, it suffices to check
this under the conditions  $|p^{it} + p^{-it}|=2$, $\sigma = 0$
and $p=2$, and in this case we find a value of  
$0.135$. In particular, since
$\zeta(1-2\sigma) < 0$ this proves that $\mathcal{S}_2 \leq 0$.

Next we turn to $\mathcal{S}_1$.  We have
\begin{align*}\mathcal{S}_1 &= \zeta(1+2\sigma) \frac{\zeta(2 +
2\sigma)}{\zeta(4+4\sigma)} \left|\zeta(2+2\sigma + 2it)\right|^2 {\sum_{k =
gh}}^\flat \frac{1}{g^{1 +
2\sigma}h^{2+2\sigma}}
\\&\qquad\times {\sum_{\substack{(m_1,n_1)=1\\(m_2,
n_2)=1\\(m_1n_1m_2n_2,k )=1} } }^\flat
\frac{\mu(m_1)\mu(m_2)F(m_1n_1k)F(m_2n_2k)}{m_1^{1 +2\sigma+
it}m_2^{1+2\sigma-it}n_1^{1 +
2\sigma + 2it}n_2^{1 + 2\sigma - 2it}} 
 \\&\qquad \times \prod_{p|m_1}\left(
\frac{p+1}{p}\cdot \frac{p^{it} + p^{-it}}{1 +
p^{-2-2\sigma}}\right)\prod_{p|m_2}\left(
\frac{p+1}{p}\cdot\frac{p^{it} + p^{-it}}{1 +
p^{-2-2\sigma}}\right)
\\&\qquad\times \prod_{p|(m_1,m_2)}\left(\frac{p}{p+ 1}\cdot \frac{(1 + p^{2it}
+ p^{-2it} -
p^{-2-2\sigma})(1 + p^{-2-2\sigma})}{(p^{it} + p^{-it})^2}\right)
\\
&=\zeta(1+2\sigma) \frac{\zeta(2 +
2\sigma)}{\zeta(4+4\sigma)} \left|\zeta(2+2\sigma + 2it)\right|^2
\\&\qquad\times \left(\frac{1}{2\pi i}\right)^2 \int_{(2)}\int_{(2)}
\hat{F}(\alpha)\hat{F}(\beta) G(\alpha, \beta; \sigma, t)d\alpha
d\beta\end{align*} where $G$ is given by
\begin{align*}&G(\alpha, \beta; \sigma, t)\\
& = \frac{\zeta(4 + 4\sigma)}{\zeta(2
+
2\sigma)}\prod_p \Biggl[1 + \frac{1}{p^{2 + 2\sigma}} +  \frac{1}{p^{1 + 2\sigma
+ \alpha + \beta}} + \frac{1}{p^{2 + 2\sigma + \alpha + \beta}} +
\frac{1}{p^{4 + 4\sigma + \alpha + \beta}}  
\\&\qquad- \frac{1}{p^{5 + 6\sigma + \alpha +
\beta}}- \frac{1}{p^{1 + 2\sigma + \alpha}} - \frac{1}{p^{2 + 2\sigma +
\alpha}} - \frac{1}{p^{2 + 2 \sigma + 2it + \alpha}}
\\&\qquad + \frac{1}{p^{3 + 4\sigma +
2it + \alpha}} - \frac{1}{p^{1 + 2\sigma + \beta}}- \frac{1}{p^{2 + \sigma +
\beta}}- \frac{1}{p^{2 + 2\sigma -2it + \beta}} + \frac{1}{p^{3 + 4\sigma -2it +
\beta}}\Biggr]
\end{align*}
Here \[G(\alpha, \beta; \sigma, t) = \frac{\zeta(1 + 2\sigma + \alpha + \beta
)}{\zeta(1+2\sigma +\alpha)\zeta(1 + 2\sigma + \beta)} \tilde{H}(\alpha, \beta;
\sigma, t)\] where $\tilde{H}$ is given by an absolutely convergent Euler
product for \[\min(\Re(\alpha), \Re(\beta), \Re(\alpha + \beta)) > -2\sigma -
c\] for some $c > 0$.  This is to say that the contour giving
$\mathcal{S}_1$ under the natural average is the same as for the harmonic
average up to a change in the absolutely convergent Euler product.  Thus the
analysis of $\mathcal{S}_1$ from the previous section goes through without
change to give
\begin{align*}\mathcal{S}_1 &= \zeta(1+2\sigma) \frac{\zeta(2 +
2\sigma)}{\zeta(4+4\sigma)} \left|\zeta(2+2\sigma + 2it)\right|^2 G(0,0;\sigma,
t) + O(k^{-\theta\sigma}) \\&= 1 + O(k^{-\theta\sigma}).\end{align*}

The reader may check that the contour integral giving $\mathcal{S}_3$ is the
same for the natural average as for the harmonic average, up to an absolutely
convergent Euler product.  Thus the analysis of the previous section yields the
bound $\Re(\mathcal{S}_3) = O(k^{-\theta \sigma})$, which completes the proof
of Proposition \ref{truncated}.

\subsection*{Acknowledgements}
It is a pleasure to thank my PhD adviser Kannan Soundararajan for his guidance, and also to thank Xiannan Li for helpful conversations and encouragement.

\end{document}